\documentclass[a4paper, USenglish, cleveref, autoref, thm-restate]{lipics-v2021}

\usepackage{amsmath,amssymb}
\usepackage{mathabx}
\usepackage{subcaption}
\usepackage{lineno}
\usepackage[usenames,dvipsnames]{xcolor}
\usepackage[nocompress,noadjust]{cite}
\usepackage{microtype}
\usepackage{xspace}
\usepackage[basic]{complexity}

\newcommand{\facetype}[4]{[#1,#2,#3,#4]\xspace}
\newcommand{\goodconfig}{bad configuration\xspace}
\newcommand{\mixedconfig}{mixed configuration\xspace}

\newcounter{inlineenum}
\renewcommand{\theinlineenum}{\roman{inlineenum}}
\newenvironment{inlineenum}
  {\unskip\ignorespaces\setcounter{inlineenum}{0}%
   \renewcommand{\item}{\refstepcounter{inlineenum}{({\theinlineenum})~}}}
  {\ignorespacesafterend}
\pdfoutput=1 
\hideLIPIcs 
\graphicspath{{./figures/}}

\title{Linear Layouts of Bipartite Planar Graphs}

\author{Henry Förster}{Department of Computer Science, University of Tübingen, Tübingen, Germany}{henry.foerster@uni-tuebingen.de}{https://orcid.org/0000-0002-1441-4189}{}
\author{Michael Kaufmann}{Department of Computer Science, University of Tübingen, Tübingen, Germany}{michael.kaufmann@uni-tuebingen.de}{https://orcid.org/0000-0001-9186-3538}{}
\author{Laura Merker}{Institute of Theoretical Informatics, Karlsruhe Institute of Technology, Germany}{laura.merker2@kit.edu}{https://orcid.org/0000-0003-1961-4531}{}
\author{Sergey Pupyrev}{Meta Platforms Inc., Menlo Park, CA, USA}{spupyrev@gmail.com}{https://orcid.org/0000-0003-4089-673X}{}
\author{Chrysanthi N. Raftopoulou}{National Technical University of Athens, Greece}{crisraft@mail.ntua.gr}{https://orcid.org/0000-0001-6457-516X}{}

\authorrunning{H. Förster, M. Kaufmann, L. Merker, S. Pupyrev, C. N. Raftopoulou}
\Copyright{Henry Förster, Michael Kaufmann, Laura Merker, Sergey Pupyrev, Chrysanthi Raftopoulou}

\relatedversion{An extended abstract of this paper appears in the Proceedings of the 18th International Symposium on Algorithms and Data Structures, \textsc{wads} 2023} 
\acknowledgements{This research was initiated at the \textsc{gnv} workshop in Heiligkreuztal, Germany, June 26 -- July 1, 2022, organized by Michalis Bekos and Michael Kaufmann.
Thanks to the organizers and other participants for creating a productive environment.}

\renewcommand{\leq}{\leqslant}
\renewcommand{\geq}{\geqslant}
\DeclareMathOperator{\qn}{qn}
\DeclareMathOperator{\sn}{sn}

\newcommand{\df}[1]{{\it #1}}

\usepackage{thm-restate}

\begin{document}
\maketitle              
\begin{abstract}
A linear layout of a graph $ G $ consists of a linear order $\prec$ of the vertices and a partition of the edges. 
A part is called a \emph{queue} (\emph{stack}) if no two edges nest (cross), that is, two edges $ (v,w) $ and $ (x,y) $ with $ v \prec x \prec y \prec w $ ($ v \prec x \prec w \prec y $) may not be in the same queue (stack).
The best known lower and upper bounds for the number of queues needed for planar graphs are 4 [Alam et al., Algorithmica 2020] and 42 [Bekos et al., Algorithmica 2022], respectively.
While queue layouts of special classes of planar graphs have received increased attention following the breakthrough result of [Dujmović et al., J.\ ACM 2020], the meaningful class of bipartite planar graphs has remained elusive so far, explicitly asked for by Bekos et al. In this paper we investigate bipartite planar graphs and give an improved upper bound of 28 by refining existing techniques.
In contrast, we show that two queues or one queue together with one stack do not suffice; the latter answers an open question by Pupyrev [GD 2018].
We further investigate subclasses of bipartite planar graphs and give improved upper bounds; in particular we construct 5-queue layouts for 2-degenerate quadrangulations.
\end{abstract}

\section{Introduction}

Since the 1980s, linear graph layouts have been a central combinatorial problem in topological graph theory, with a wealth of publications \cite{DBLP:journals/jct/BernhartK79, CLR87, DBLP:journals/jcss/Yannakakis89, HR92, DBLP:journals/siamdm/HeathLR92, DMW05, DBLP:journals/algorithmica/BekosBKR17, DBLP:journals/jgaa/DujmovicF18, DJMMUW20, BGR22}. A \emph{linear layout} of a graph consists of a linear order $\prec$ of the vertices and a partition of the edges into \emph{stacks} and \emph{queues}. A part is called a \emph{queue} (\emph{stack}) if no two edges of this part nest (cross), that is, two edges $ (v,w) $ and $ (x,y) $ with $ v \prec x \prec y \prec w $ ($ v \prec x \prec w \prec y $) may not be in the same queue (stack).
Most notably, research has focused on so-called \emph{stack layouts} (also known as book-embeddings) and \emph{queue layouts} where  either all parts are \emph{stacks} or all parts are \emph{queues}, respectively. While these kinds of graph layouts appear quite restrictive on first~sight, they are in fact quite important in practice. For instance, stack layouts are used as a model for chip design~\cite{CLR87}, while queue layouts find applications in three-dimensional network visualization~\cite{DBLP:journals/jgaa/BiedlSWW99,DMW05,DBLP:conf/gd/DujmovicW03}. For these applications, it is important that the edges are partitioned into as few stacks or queues as possible. This notion is captured by the \emph{stack number}, $\sn(G)$, and \emph{queue number}, $\qn(G)$, of a graph $G$, which denote how many stacks or queues are required in a stack and a queue layout of $G$, respectively.
Similarly, \emph{mixed linear layouts}, where both stacks and queues are allowed, have emerged as a research direction in the past few years~\cite{ABKM20,CKN19,Pup18}. 

Recently, queue layouts have received much attention as several breakthroughs were made which pushed the field further. Introduced in 1992~\cite{HR92}, it was conjectured in the same year~\cite{DBLP:journals/siamdm/HeathLR92}, that all planar graphs have a bounded queue number. Despite various attempts at settling the conjecture~\cite{BDDEW15,DBLP:journals/siamcomp/BattistaFP13,DBLP:journals/jgaa/DujmovicF18,DBLP:journals/siamcomp/BekosFGMMRU19}, it remained unanswered for almost 30 years. In 2019, the conjecture was finally affirmed by Dujmovi\'c, Joret, Micek, Morin, Ueckerdt and Wood~\cite{DJMMUW20}. Their proof relies on three ingredients: First, it was already known that graphs of bounded treewidth have bounded queue number~\cite{Wie17}. Second, they showed that the \emph{strong product}  of a graph of bounded queue number and a path has bounded queue number. Finally, and most importantly, they proved that every planar graph is a subgraph of the \emph{strong product} of a path and a graph of 
treewidth at most $8$.  In the few years following the result, both queue layouts~\cite{DBLP:conf/gd/AlamBG0P20,DBLP:journals/jgaa/BhoreGMN22,BGR22,DBLP:journals/combinatorica/DujmovicEHMW22,DBLP:conf/gd/MerkerU20} and graph product structure~\cite{DBLP:journals/corr/abs-2204-11495,DBLP:conf/swat/BoseMO22,DBLP:journals/corr/abs-2206-02395,DBLP:journals/algorithmica/Morin21,DBLP:journals/combinatorics/UeckerdtWY22,DBLP:journals/corr/abs-2208-10074} have become important research directions. Yet, after all recent developments, the best known upper bound for the queue number of planar graphs is 42~\cite{BGR22} whereas the best known corresponding lower bound is 4~\cite{ABGKP18}. This stands in contrast to a tight bound of 4 for the stack number of planar graphs~\cite{DBLP:journals/jocg/KaufmannBKPRU20,DBLP:journals/jcss/Yannakakis89}.

It is noteworthy that better upper bounds of the queue number are known only for certain subclasses of planar graphs, such as planar $3$-trees~\cite{ABGKP18} and posets~\cite{DBLP:conf/gd/AlamBG0P20}, or for relaxed variants of the queue number~\cite{DBLP:conf/gd/MerkerU20}. It remains elusive how other properties of a graph, such~as a bounded degree or bipartiteness, can be used to 
reduce the gap between the lower and the upper bounds on the queue number; see also the open problems raised in~\cite{BGR22} which contains the currently best upper bound. This is partially due to the fact that it is not well understood how these properties translate into the product structure of the associated graph classes. In fact, the product structure theorem has been improved for general planar graphs~\cite{DBLP:journals/combinatorics/UeckerdtWY22} while, to the best of our knowledge, there are very few results  that yield stronger properties for subclasses thereof.
Here, we contribute to this line of research by studying bipartite planar~graphs.

\subparagraph{Results.}
Our paper focuses on the queue number of bipartite graphs and subclasses thereof.
We start by revisiting results from the existing literature in Section~\ref{sec:prel}. 
In particular, we discuss techniques that are used to bound the queue number of general planar graphs by 42~\cite{DJMMUW20,BGR22} and refine them to obtain an improved upper bound on the queue number of bipartite planar graphs. 

\begin{restatable}{theorem}{qnBipartiteUB}
    \label{thm:qn-bipartite-ub}
    The queue number of bipartite planar graphs is at most $28$.
\end{restatable}

We then improve this bound for interesting subfamilies of bipartite planar graphs. 
For this we first prove a product structure theorem for stacked quadrangulations, which is a family of graphs that may be regarded as a bipartite variant of planar 3-trees.
We remark that we avoid the path factor that is present in most known product structure theorems.

\begin{restatable}{theorem}{productStackedQuadrangulations}
\label{thm:product-stacked-quadrangulations}
    Every stacked quadrangulation is a subgraph of $ H \boxtimes C_4 $, where $ H $ is a planar $ 3 $-tree.
\end{restatable}

In fact, our result generalizes to similarly constructed graph classes.
Based on \cref{thm:product-stacked-quadrangulations}, we improve the upper bound on the queue number of stacked quadrangulations.

\begin{restatable}{theorem}{qnStackedQuadrangulationsUB}
\label{thm:qn-stacked-quadrangulations-ub}
    The queue number of stacked quadrangulations is at most $ 21 $.
\end{restatable}

Complementing our upper bounds, we provide lower bounds on the queue number and the mixed page number of bipartite planar graphs in \cref{sec:lower-bounds}.
Both results improve the state-of-the-art for bipartite planar graphs, while additionally providing a lower bound for the special case of 2-degenerate bipartite planar graphs.
We remark that \cref{thm:mixed-lb} answers a question asked in \cite{Pup18,CKN19,ABKM20}.

\begin{restatable}{theorem}{qnBipartiteLB}
\label{thm:qn-bipartite-lb}
    There is a $2$-degenerate bipartite planar graph with queue number at least~$ 3 $.
\end{restatable}

\begin{restatable}{theorem}{mixedLB}
\label{thm:mixed-lb}
    There is a $2$-degenerate bipartite planar graph that does not admit a $ 1 $-queue $ 1 $-stack layout.
\end{restatable}

For this purpose, we use a family of $2$-degenerate quadrangulations. Finally, inspired by our lower bound construction, we conclude with investigating this graph class.

\begin{restatable}{theorem}{degenerate}
    \label{thm:qn-2-degenerate-ub}
    Every $2$-degenerate quadrangulation admits a $5$-queue layout.
\end{restatable}

\subparagraph{Outline.}

We start with basic results on the queue number of bipartite planar graphs in \cref{sec:prel}, where we prove \cref{thm:qn-bipartite-ub}.
\Cref{sec:product-structure} provides a definition of stacked quadrangulations and an investigation of their structure including proofs of \cref{thm:product-stacked-quadrangulations,thm:qn-stacked-quadrangulations-ub}.
We continue with lower bounds in \cref{sec:lower-bounds} and then further investigate the graphs constructed there in \cref{sec:2-degenerate}, in particular we prove \cref{thm:qn-2-degenerate-ub}.

\section{Preliminaries}
\label{sec:prel}

In this section, we introduce basic definitions and tools that we use to analyze  
the queue number of bipartite planar graphs and refine them to prove \cref{thm:qn-bipartite-ub}.

\subsection{Definitions}
\label{sec:definitions}


\subparagraph{Classes of bipartite planar graphs.} 
In this paper, we study subclasses of \emph{planar} graphs, that is graphs admitting a \emph{planar drawing}. 
A special type of planar drawings are \df{leveled planar drawings} where the vertices are placed on a sequence of parallel lines (levels) and every edge joins 
vertices in two consecutive levels. We call a graph  \df{leveled planar} if it admits a leveled planar drawing. 

A planar drawing partitions the plane into  regions called \emph{faces}.
It is well known that the \emph{maximal} planar graphs, that is, the planar graphs to which no crossing-free edge can be added, are exactly the triangulations of the plane, that is, every face is a triangle. 
We focus on the \emph{maximal bipartite planar graphs} which are exactly the quadrangulations of the plane, that is, every face is a quadrangle. In the following, we introduce  interesting families of quadrangulations.

One such family are the \emph{stacked quadrangulations} that can be constructed as follows.
First, a square is a stacked quadrangulation.
Second, if $ G $ is a stacked quadrangulation and $ f $ is a face of $ G $, then inserting a plane square $ S $ into $ f $ and connecting the four vertices of $ S $ with a planar matching to the four vertices of $ f $ again yields a stacked quadrangulation. Note that every face has four vertices, that is, the constructed graph is indeed a quadrangulation. We are particularly interested in this family of quadrangulations as stacked quadrangulations can be regarded as the bipartite variant of the well-known graph class \emph{planar $3$-trees} which are also known as  \emph{stacked triangulations}. This class again can be recursively defined as follows: A \emph{planar $3$-tree}  is either a triangle or a graph that can be obtained from a planar $3$-tree by adding a vertex $ v $ into some face $ f $ and connecting $ v $ to the three vertices of $f  $. This class is particularly interesting in the context of queue layouts as it provides the currently best lower bound on the queue number of planar graphs~\cite{ABGKP18}.

The notions of stacked triangulations and stacked quadrangulations can be generalized as follows using once again a recursive definition. For $ t \geq 3 $ and $ s \geq 1 $, any connected planar graph of order at most $ s $ is called a {$(t,s)$-stacked graph}. Moreover, for a $(t,s)$-stacked graph $ G $ and a connected planar graph $ G'$ with at most $s$ vertices, we obtain another $(t,s)$-stacked graph $ G'' $ by connecting $ G' $ to the vertices of a face $ f $ of $G$ in a planar way  such that each face of $G''$ has at most $ t $ vertices. Note that we do not require that the initial graph $G$ and the connected stacked graph $G'$ of order at most $s$ in the recursive definition have only faces of size at most $t$ as this will not be required by our results in Section~\ref{sec:product-structure}.
If in each recursive step the edges between $ G' $ and the vertices of $ f $ form a matching, then the resulting graph is called an \emph{$ (t, s) $-matching-stacked graph}.
Now, in particular $ (3, 1) $-stacked graphs are the  {planar $3$-trees} while  the {stacked quadrangulations}  are a subclass of the $ (4, 4) $-matching-stacked graphs. 
In addition, $ (3, 3) $-stacked graphs are the \emph{stacked octahedrons}, which were successfully used to construct planar graphs that require four stacks~\cite{DBLP:journals/jocg/KaufmannBKPRU20,DBLP:journals/jctb/Yannakakis20}.

In addition to graphs obtained by recursive stacking operations, we will also study graphs that are restricted by \emph{degeneracy}. Namely, we call a graph $G=(V,E)$ $d$-degenerate if there exists a total order $(v_1,\ldots,v_n)$ of $V$, so that for $1 \leq i \leq n$, $v_i$ has degree at most $d$ in the subgraph induced by vertices $(v_1,\ldots,v_i)$. Of particular interest will be a family of $2$-degenerate quadrangulations discussed in Section~\ref{sec:lower-bounds}. It is worth pointing out that there are recursive constructions for triconnected and simple quadrangulations that use the insertion of degree-$2$ vertices and $(4,4)$-matching stacking in their iterative steps~\cite{DBLP:journals/dmtcs/FelsnerHKO10}.

\subparagraph{Linear layouts.} A \emph{linear layout} of a graph $G =(V,E)$ consists of an order $\prec$ of $V$ and a partition $\mathcal{P}$ of $E$. Consider two disjoint edges $(v,w),(x,y) \in E$. We say that $(v,w)$ \emph{nests} $(x,y)$ if $ v \prec x \prec y \prec w $, 
and we say that $(v,w)$ and $(x,y)$ \emph{cross} if $ v \prec x \prec w \prec y $. 
For each part $P \in \mathcal{P}$ we require that either no two edges of $P$ nest or that no two edges of $P$ cross. 
We call $P$ a \emph{queue} in the former and a \emph{stack} in the latter case. Let $\mathcal{Q} \subseteq \mathcal{P}$ denote the set of queues and let $\mathcal{S} \subseteq \mathcal{P}$ denote the set of stacks; such a linear layout is referred to as a
\emph{$|\mathcal{Q}|$-queue $|\mathcal{S}|$-stack} layout. 
If $\mathcal{Q} \neq \emptyset$ and $\mathcal{S} \neq \emptyset$, we say that the linear layout is \emph{mixed}, while 
when $\mathcal{S}=\emptyset$, it is called a \emph{$|\mathcal{Q}|$-queue} layout. The \emph{queue number}, $\qn(G)$, of a graph $G$ 
is the minimum value $q$ such that $G$ admits a $q$-queue layout. Heath and Rosenberg~\cite{HR92} 
characterize graphs admitting a $1$-queue layout in terms of \emph{arched-leveled} layouts. 
In particular, this implies that each \emph{leveled planar graph} admits a $1$-queue layout. 
Indeed a bipartite graph has queue number 1 if and only if it is leveled planar~\cite{AG11}.

In the remainder of this subsection, we define important tools that have been used in the context of linear layouts in the past and also are essential in our proofs.

\subparagraph{Important tools.}
The \emph{strong product} $ G_1 \boxtimes G_2 $ of two graphs $ G_1 $ and $ G_2 $ is a graph with vertex set $ V(G_1) \times V(G_2) $ and an edge between two vertices $ (v_1, v_2) $ and $ (w_1, w_2) $ if \begin{inlineenum}
    \item  $ v_1 = w_1 $ and $(v_2,w_2) \in E(G_2) $,
    \item $ v_2 = w_2 $ and $(v_1,w_1) \in E(G_1) $, or 
    \item $(v_1,w_1) \in E(G_1) $ and $(v_2,w_2) \in E(G_2) $.
\end{inlineenum}

Given a graph $G$, an \df{$H$-partition} of $G$ is a pair $\big(H, \{V_x : x\in V(H)\}\big)$ consisting of a 
graph $H$ and a partition of $V$ into sets $\{V_x : x\in V(H)\}$ called \emph{bags}
such that for every edge $(u, v) \in E$ one of the following holds:
\begin{inlineenum}
\item\label{def:intrabag} $u, v \in V_x$ for some $x \in V(H)$, or 
\item\label{def:interbag}~there is an edge $(x, y)$ of $H$ with $u \in V_x$ and $v \in V_y$.
\end{inlineenum} In Case~\eqref{def:intrabag}, we call $(u,v)$ an \emph{intra-bag} edge, while we call it \emph{inter-bag} in Case~\eqref{def:interbag}.
To avoid confusion with the vertices of $G$, the vertices of $H$ are called \df{nodes}. The \emph{width} of an $ H $-partition is defined as the maximum size of a bag. 
It is easy to see that a graph $ G $ has an $ H $-partition of width $ w $ if and only if it is a subgraph of $ H \boxtimes K_w $ (compare Observation~35 of \cite{DJMMUW20}).
In the case where $H$ is a tree, we call the $H$-partition a \emph{tree-partition}.
We refer to \cref{fig:defs} for an example of a tree-partition.

A related concept to $H$-partitions are \emph{tree decompositions}. 
Given a graph $G$, a \df{tree decomposition} of $G$ is a pair $\big(T, \{B_x : x\in V(T)\}\big)$ consisting of a tree $T$ where every vertex $x \in V(T)$ is associated with a subset $B_x$, called \emph{bag}, of the vertices of $G$ so that the following hold:
\begin{inlineenum}
\item\label{def:tree-decomposition:1} $\bigcup_{x \in V(T)} B_x= V(G)$, 
\item\label{def:tree-decomposition:2} for each $(u,v) \in E(G)$ there exists at least one bag $B_x$ so that $u \in B_x$ and $v \in B_x$, and 
\item\label{def:tree-decomposition:3} for every $v \in V(G)$, the set of nodes whose bags contain $v$ induce a connected subtree of $T$.
\end{inlineenum}
Observe that in contrast to a tree-partition, the bags are in general not disjoint. The \emph{treewidth} of a tree decomposition is the cardinality of the largest bag minus one. Moreover, the \emph{treewidth} of a graph is the minimum treewidth of any of its tree decompositions.

\subsection{General Bipartite Planar Graphs}
\label{sec:general_bipartite}

Following existing works on queue layouts of planar graphs, we get the following upper bound for the class of bipartite planar graphs.

\qnBipartiteUB*


A key component of the state-of-the-art bounds on the queue number of planar graphs are $H$-partitions with low layered width.
A \emph{BFS-layering} of $G$ is a partition of $V(G)$ into $\mathcal{L}=(V_0,V_1,\ldots)$ such that $V_i$ contains exactly the vertices with graph-theoretic distance $i$ from a specified vertex $r \in V(G)$. We refer to each $V_i$ as a \emph{layer}. We say that an $H$-partition $(H,\{A_x:x \in V(H)\})$ of a graph $G$ has \emph{layered-width} $\ell$ if and only if there is a  {BFS-layering} $\mathcal{L}=(V_0,V_1,\ldots)$ of $G$ so that for every $x \in V(H)$ and every $i$ it holds that $|A_x \cap V_i| \leq \ell$. The following theorem plays an important role in computing queue layouts for planar graphs with a constant number of queues.

\begin{claim}[Theorem 15 of \cite{DJMMUW20}]\label{thm:djmmuwBfsLayering}
Every planar graph $G$ has an $H$-partition with layered-width $3$ such that $H$ is planar and has treewidth at most 3. Moreover, there is such a partition for every BFS layering of
$G$.
\end{claim}

There are two noteworthy observations regarding \cref{thm:djmmuwBfsLayering}. First, for a bipartite planar graph with parts $A$ and $B$ so that $V=A \dot{\cup} B$, a BFS-layering $\mathcal{L}=(V_0,V_1,\ldots)$ is so that without loss of generality for every integer $k$, it holds $V_{2k} \cap B = \emptyset$ and $V_{2k+1} \cap A = \emptyset$. We will call such a layering \emph{bichromatic} and we say that an $H$-partition $\{A_x|x \in V(H)\}$ of a bipartite graph $G$ has \emph{bichromatic layered-width} $\ell$ if and only if there is a bichromatic BFS-layering $\mathcal{L}=(V_0,V_1,\ldots)$ of $G$ so that for every $x \in V(H)$ and every $i$ it holds that $|A_x \cap V_i| \leq \ell$. 

Second, on the other hand, the proof of \cref{thm:djmmuwBfsLayering} assumes $G$ to be triangulated. Hence, it is not immediate, that the following special case of \cref{thm:djmmuwBfsLayering} is true:

\begin{lemma}\label{thm:bichromaticBfsLayering}
 Every bipartite planar graph $G$ has a  $H$-partition with bichromatic layered-width $3$ such that $H$ is planar and has treewidth at most 3.
\end{lemma}

\begin{proof}
   We begin by quadrangulating the input graph $G$ obtaining the quadrangulation $G'$. Then, we perform a BFS traversal of $G'$. In the resulting layering, every quadrangle $q=(a_1,b_1,a_2,b_2)$ is (up to relabeling) either such that $a_1 \in V_i$, $b_1,b_2, \in V_{i+1}$ and $a_2 \in V_{i+2}$ or such that $a_1,a_2 \in V_i$, $b_1,b_2, \in V_{i+1}$. In either case, we can triangulate $q$ with edge $(b_1,b_2)$. Applying this procedure to all quadrangles yields a triangulated supergraph $G''$ that has the property that the layering obtained from $G'$ still is a valid layering for $G''$. Then, we can apply \cref{thm:djmmuwBfsLayering} to obtain the $H$-partition for $G''$ which serves as the required bichromatic $H$-partition for graph $G$.
 \end{proof}

Based on this decomposition, Dujmović et al.~\cite{DJMMUW20} compute a queue-layout using the following lemma:

\begin{claim}[Lemma 8 of \cite{DJMMUW20}]\label{lem:djmmuwQueueLayout}
For all graphs $H$ and $G$, if $H$ has a $k$-queue layout and $G$ has an $H$-partition of
layered-width $\ell$ with respect to some layering $(V_0, V_1, \ldots)$ of $G$, then $G$ has a $(3\ell k + \lfloor \frac{3}{2} \ell\rfloor )$-queue
layout using vertex order $\overrightarrow{V_0},\overrightarrow{V_1},\ldots$, where $\overrightarrow{V_i}$ is some order of $V_i$. In particular,\begin{equation}\label{eq:queueEq}
   \qn(G) \leq 3 \ell\  \qn(H) + \left\lfloor \frac{3}{2}\ell \right\rfloor.
\end{equation}
\end{claim}

In order to improve the lemma for bipartite planar graphs, we briefly sketch which components the upper bound of $\qn(G)$ in~\eqref{eq:queueEq} is composed of.

\begin{proof}[Sketch of the proof of \cref{lem:djmmuwQueueLayout}]
   The main argument here, is to classify edges as \emph{intra-bag} or \emph{inter-bag} as well as as \emph{intra-layer} or \emph{inter-layer}. Namely, an edge is called \emph{intra-bag} if its two endpoints occur in the same bag $A_x$ with $x \in V(H)$. Otherwise, it is called \emph{inter-bag}. Similarly, an edge is called \emph{intra-layer} if both its endpoints occur on the same layer $V_i$, otherwise it is called \emph{inter-layer}.
   
   For each classification of edges, different queues are used. More precisely:
   \begin{enumerate}[E.1]
       \item\label{e:intraLayer1} Intra-layer intra-bag edges may induce a $K_\ell$, hence there are at most $\lfloor \frac{\ell}{2} \rfloor$ queues needed for such edges.
       \item\label{e:inter0} Inter-layer intra-bag edges may induce a $K_{\ell,\ell}$, hence there are at most $\ell$ queues needed for such edges.
       \item\label{e:intraLayer2} Intra-layer inter-bag edges corresponding to the same edge of $H$ may induce a $K_{\ell,\ell}$, hence there are at most $\ell k$ queues needed for such edges as $H$ has queue number at most $k$.
       \item\label{e:inter1} Inter-layer inter-bag \emph{forward}\footnote{Let $(u,v)$ be an inter-layer inter-bag edge so that $u \in A_x$ and $v \in A_y$. Then, $(u,v)$ is called \emph{forward} if and only if $x$ precedes $y$ in the $k$-queue layout of $H$. Otherwise, it is called \emph{backward}.} edges corresponding to the same edge of $H$ may induce a $K_{\ell,\ell}$, hence there are at most $\ell k$ queues needed for such edges as $H$ has queue number at most $k$.
       \item\label{e:inter2} Inter-layer inter-bag \emph{backward} edges corresponding to the same edge of $H$ may induce a $K_{\ell,\ell}$, hence there are at most $\ell k$ queues needed for such edges as $H$ has queue number at most $k$.
\end{enumerate}  This yields at most $\lfloor \frac{3}{2}\ell \rfloor$ intra-bag edges and at most $3 \ell k$ inter-bag edges.  
 \end{proof}

Given a $H$-partition with bounded bichromatic layered-width, \cref{lem:djmmuwQueueLayout} can obviously be refined as follows:

\begin{lemma}\label{lem:bichromaticQueueLayout}
For all graphs $H$ and $G$, if $H$ has a $k$-queue layout and $G$ has a $H$-partition of
bichromatic layered-width $\ell$ with respect to some bichromatic layering $(V_0, V_1, \ldots)$ of $G$, then $G$ has a $(2\ell k + \ell)$-queue
layout using vertex order $\overrightarrow{V_0},\overrightarrow{V_1},\ldots$, where $\overrightarrow{V_i}$ is some order of $V_i$. In particular,\begin{equation}\label{eq:queueEqBipartite}
   \qn(G) \leq 2 \ell\  \qn(H) + \ell.
\end{equation}
\end{lemma}

\begin{proof}
   Since the layering is bichromatic, we observe that $G$ contains no intra-layer edges. These are accounted with at most $\lfloor \frac{\ell}{2} \rfloor$ intra-bag queues~(E.\ref{e:intraLayer1}) and at most $k \ell$ inter-bag queues~(E.\ref{e:intraLayer2}) in the proof of \cref{lem:djmmuwQueueLayout}. Hence, the statement follows.
 \end{proof}

\cref{thm:bichromaticBfsLayering} and \cref{lem:bichromaticQueueLayout} already imply that every planar bipartite graph $G$ has a $33$-queue layout (as planar 3-trees have queue number at most $5$).  This bound can be reduced to $28$ following the modifications of Bekos et al.~\cite{BGR22}. Namely, Bekos et al. apply the following modifications:

\begin{enumerate}[B.1]
   \item\label{bgr:1} Additional constraints are maintained for a proper layered drawing algorithm for outerplanar graphs.
   \item\label{bgr:2} Based on (B.\ref{bgr:1}), additional constraints for the $5$-queue layout of planar $3$-trees are shown.
   \item\label{bgr:3} \emph{Degenerate tripods}\footnote{ A \emph{tripod} is the content of a bag of $H$. It consists of a triangle from which three \emph{vertical paths} (traversing the BFS-layering in order) emerge (and potentially some edges between such paths). In a \emph{degenerate} tripod, one of the paths consists of a single vertex.} in the $H$-decomposition of $G$ are excluded by inserting a three new vertices inside each face $f$ of the triangulated graph $G$ and triangulating appropriately. In particular, the three new vertices inside $f$ become leafs in the augmented BFS-tree, that is, the BFS-layering of the original graph $G$ stays intact. 
   \item\label{bgr:4} Based on (B.\ref{bgr:1}) to (B.\ref{bgr:3}), the order of vertices within each  layer is chosen more carefully when applying \cref{lem:djmmuwQueueLayout}. 
\end{enumerate}

As a result, Bekos et al.\ obtain the following:

\begin{claim}[Lemmas 5 and 6 of \cite{BGR22}]\label{lem:bgr}
In the queue layout computed by the modifications~(B.\ref{bgr:1}) to~(B.\ref{bgr:4}) of the algorithm of Dujmović et al.~\cite{DJMMUW20}, the intra-bag inter-layer edges (see E.\ref{e:inter0}) can be assigned to at most $2$ queues while each set of inter-bag edges (see (E.\ref{e:intraLayer2}) to (E.\ref{e:inter2})) can be assigned to at most $13$ queues. 
\end{claim}

We are now ready to prove \cref{thm:qn-bipartite-ub}:

\begin{proof}[Proof of \cref{thm:qn-bipartite-ub}] By \cref{thm:bichromaticBfsLayering}, for every bipartite planar graph $G$ there is a $H$-partition with bichromatic layered width $3$ such that $H$ is planar and has treewidth at most $3$. We then apply \cref{lem:bichromaticQueueLayout,lem:bgr}. In particular, we observe that \cref{lem:bgr} is applicable since none of the modifications described in Steps (B.\ref{bgr:1}) to (B.\ref{bgr:4}) interfere with our bichromatic layering. Namely, (B.\ref{bgr:1}) and (B.\ref{bgr:2}) directly operate on $H$ while (B.\ref{bgr:3}) and (B.\ref{bgr:4}) respect the layering. Hence, by \cref{lem:bgr}, for bipartite planar graphs we need $2$ queues for intra-bag inter-layer edges (E.\ref{e:inter0}) and in total $26$ queues for inter-bag inter-layer edges (E.\ref{e:inter1} and E.\ref{e:inter2}) resulting in $28$ queues overall.  
\end{proof}

\section{Structure of Stacked Quadrangulations}
\label{sec:product-structure}

In this section we investigate the structure of stacked quadrangulations and then deduce an upper bound on the queue number.
In particular we show that every stacked quadrangulation is a subgraph of the strong product $ H \boxtimes C_4 $, where $H$ is a planar $3$-tree. 
Recall that stacked quadrangulations are $(4,4)$-matching-stacked graphs. 
For ease of presentation, we first prove a product structure theorem for general $ (t, s) $-stacked graphs as it prepares the proof for the matching variant. 

\begin{restatable}{theorem}{productStacked}
    \label{thm:product_fs-stacked}
    For $ t \geq 3, s \geq 1 $, every $ (t, s) $-stacked graph $G$ is a subgraph of $ H \boxtimes K_s $ for some planar graph $ H $ of treewidth at most $ t $.
\end{restatable}

\begin{proof}
    In order to prove the theorem, we need to find an $H$-partition of $G$, and a tree decomposition $T$ of $H$ with treewidth at most $t$, where $H$ is a planar graph. 
    Given a $ (t, s) $-stacked graph $ G $ for $ t \geq 3 $ and $ s \geq 1 $ and its construction sequence, we define the $ H $-partition $\big(H, \{V_x : x\in V(H)\}\big)$  of width at most $ s $ as follows.
    In the base case we have a graph of size at most $ s $ whose vertices all are assigned to a single bag $ V_{x_0}$.
    Then in each construction step we add a new bag, say $V_{v'}$ that contains all new vertices (those of the inserted graph $G'$). Thus each bag contains at most $ s $ vertices, that is, the width of the $H$-partition  is at most $ s $. For $u \in V_x$ and $v \in V_y$, if $G$ contains edge $(u,v)$ then we connect $x$ and $y$ in $H$. 

    Next, we define a tree decomposition  $\big(T, \{B_x : x\in V(T)\}\big)$ of $ H $.
    We again give the definition iteratively following the construction sequence of $ G $.
    We thereby maintain as an invariant that for each face $ f $ of the current subgraph of $ G $, there is a bag $B_{y_f}$ in $ T $ containing all vertices $x$ of $ H $ for which $V_x$ contains a vertex of $ f $.
    In the base case, the tree decomposition contains a single bag that only contains node $x_0$.
    Now  consider a recursive construction step, that is  a connected graph $ G' $ of order at most $ s $ is inserted into some face $ f $ yielding a new $(t,s)$-stacked graph $G''$.
    Recall that there is a node $v'$ in $H$ with $V_{v'}=V(G')$.
    Let $ F \subseteq V(H) $ denote the nodes $x$ of $ H $ such that $V_x$ contains a vertex of $ f$. By induction hypothesis, there exists a node $y_f$ in $ T$ such that  $B_{y_f}$   contains $ F $.
    We add a new node $y$ as a leaf of $y_f$ associated with bag $ B_{y} = F \cup \{ v' \} $. We observe that this procedure indeed yields a tree decomposition of $G''$. Moreover, the new bag $B_y$ contains all nodes of $H$ that contain vertices of newly generated faces.
    Also note that each bag of $ T $ contains at most $ t + 1 $ vertices, where $ t $ is the maximum size of a face in the construction process of $ G $.
    Together with the observation that $ H $ is a minor of $ G $ and thus planar, this proves \cref{thm:product_fs-stacked}.
 \end{proof}

We now extend our ideas to the following product structure theorem.
Note that we use the same $ H $-partition and show that the treewidth is at most $ t - 1 $.

\begin{restatable}{theorem}{productMatchingStacked}
    \label{thm:product-matching-stacked}
     For $ t \geq 3, s \geq 2 $, every $ (t, s) $-matching-stacked graph $G$ is a subgraph of $ H \boxtimes K_s $ for some planar graph $ H $ of treewidth at most $ t - 1 $.
\end{restatable}

\begin{proof}
   Given a $ (t, s) $-matching-stacked graph $ G $ for $ t \geq 3, s \geq 2 $, we compute its $ H $-partition and the tree decomposition $T$ of $H$ as in the proof of Theorem~\ref{thm:product_fs-stacked}.
   By construction, the width of the $ H $-partition is at most $ s $ and $ H $ is a minor of $ G $ and thus planar.
   So it suffices to show that each bag of $ T $ contains at most $ t $ nodes of $H$ (instead of $t+1$). This is clearly true for the root-node of $T$ whose bag has only one node of $H$.  Consider a step in the construction of $ G $, that is, we have a face $ f $ bounded by at most $ t $ vertices, and we place a connected subgraph $ G' $ of order at most $ s $ inside. Assume that each of the bags placed before introducing $G'$ contains at most $t$ nodes of $H$. 
   Again, $v'$ denotes the vertex of $ H $ with $V_{v'}=V(G')$, 
   $ F $ is the set of nodes whose bags contain at least one vertex of $f$, 
   and node $ y_f $ is the node whose bag $ B_{y_f} $ contains $ F $.
   Further $ y $ is the node of $ H $ with $B_{y}=F\cup \{v'\}$ that is introduced as a leaf of $ y_f $ during this step.
   
   We claim that $F$ has at most $t-1$ nodes of $H$ and therefore $B_y$ contains at most $t$ nodes. Consider the step where face $f$ was created. During that step a subgraph $G''$ was added inside some face $f'$. So, if $f$ is an interior face of $G''$, it follows that $F$ contains only one node of $H$, namely $v''$ for which $V_{v''}=V(G'')$. On the other hand, as $G''$ is connected to the boundary of $f'$ with a matching, at least two vertices of $f$ belong to $V(G'')$. Hence the remaining vertices of $f$ belong to at most $t-2$ bags associated with nodes of $H$ and thus the set $F$ has at most $t-1$ nodes as claimed. 
   We conclude that each bag of the tree decomposition contains at most $ t $ nodes of $H$.
 \end{proof}

Theorem~\ref{thm:product_fs-stacked} shows that every stacked quadrangulation is a subgraph of $ H \boxtimes K_4 $, where $H$ is a planar graph with treewidth at most $4$. However, stacked quadrangulations are  also $ (4, 4) $-matching-stacked graphs and, hence, by Theorem~\ref{thm:product-matching-stacked} they are subgraphs of  $ H \boxtimes K_4 $, where $H$ is a planar $3$-tree. 
In the following we improve this result by replacing $ K_4 $ by $ C_4 $.

\productStackedQuadrangulations*

\begin{proof}
    Given a stacked quadrangulation $ G $ and its construction sequence, we compute its $ H $-partition as described above. Note that in \cref{thm:product-matching-stacked}, we show that $H$ is a planar $3$-tree.
    It remains to show that $H$ does not only certify that $ G \subseteq H \boxtimes K_4 $ (\cref{thm:product-matching-stacked}) but also the stronger statement that $ G \subseteq H \boxtimes C_4 $.
    That is, we need to show that in order to find $ G $ in the product, we only need edges that show up in the product with $ C_4 $.
    In each step of the construction of $ G $ we insert a 4-cycle into some face $f$. Hence, for each node $x$ of $ H $ the bag $V_x$ consists of a 4-cycle denoted by $(v_0^x,v_1^x,v_2^x,v_3^x)$.
    We label the four vertices with $ 0, 1, 2, 3 $ such that for $i=0,1,2,3$ vertex $v_{i+k}^x$  is labeled $i$ for some offset $k$ (all indices and labels taken modulo 4). 
    
    As the labels appear consecutively along the $4$-cycle, the strong product allows for edges between two vertices of distinct bags if and only if their labels differ by at most 1 (mod 4).
    That is we aim to label the vertices of each inserted 4-cycle so that each inter-bag edge connects two vertices whose labels differ by at most 1.
    We even prove a slightly stronger statement, namely that offset $k$ can be chosen for a newly inserted bag $V_x$ such that the labels of any two vertices $ v $ and $ v' $ in $G$ connected by an inter-bag edge differ by exactly 1.
    Indeed, assuming this for now, we see that if $(v,v')$ is an inter-bag edge with $v \in V_x$ and $v'\in V_y$, then $(x,y) \in E(H)$. 
    For this recall that $V_x$ and $V_y$ induce two $4$-cycles. Assume that $v=v_i^x$ with label $i$ and $v'=v_j^y$ is labeled $j$, where $0\leq i,j\leq 3$. As the two labels differ by exactly 1, we have $j=i+1 \pmod 4$ or $j=i-1 \pmod 4$. Now in the strong product $H\boxtimes C_4$, vertex $v_i^x$ is connected to all of $v_{i-1}^y,v_i^y,v_{i+1}^y$, that is, the edge $(v,v')$ exists in $H\boxtimes C_4$.
    It is left to prove that the labels can indeed by chosen as claimed.
    As shown in \cref{fig:product-C_4}, starting with a 4-cycle with labels $ 0,1,2,3 $ from the initial $4$-cycle, we obtain three types of faces that differ in the labeling of their vertices along their boundary:
    Type \textbf{(a)} has labels $ i, i+1, i+2, i+3 $, ($0\leq i\leq 3$),
    Type \textbf{(b)} has labels 
    $ i, i+1, i+2, i+1 $, ($0\leq i\leq 3$), and
    Type \textbf{(c)} has labels $ i, i+1, i, i+1 $ ($0\leq i\leq 3$).
    In each of the three cases we are able to label the new 4-cycle so that the labels of the endpoints of each inter-bag edge differ by exactly 1, see \cref{fig:product-C_4} and also \cref{fig:product_stacked_quadrangulation_example} for an example. 
    This concludes the proof. \end{proof}
    \begin{figure}[t]
            \centering
            \includegraphics{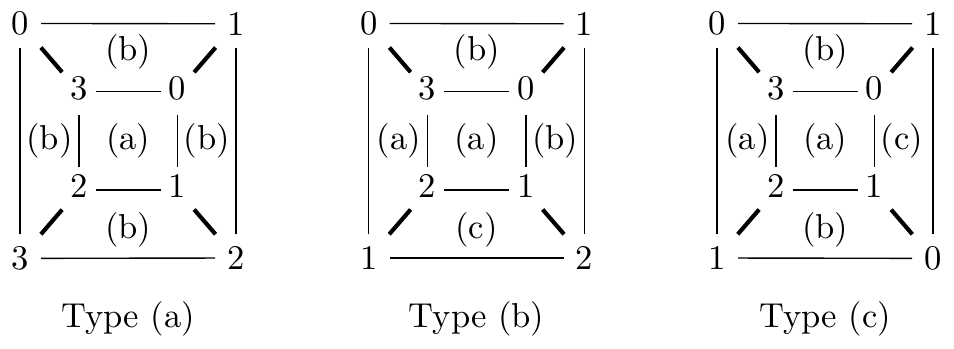}
            \caption{%
                The three types of faces defined in the proof of \cref{thm:product-stacked-quadrangulations}, each with a 4-cycle stacked inside, where the labels represent the position of the vertex in the 4-cycle.
                For better readability, vertices are represented by their labels, where a vertex with label $ i + j $ (same $ i \in \{0, 1, 2, 3\} $ for all eight vertices) is written as $ j $ (labels taken mod~4).
                That is, each type represents four situations that can occur in a face, e.g., the other three cases of Type (b) have labels 1, 2, 3, 2 ($ i = 1 $), resp. 2, 3, 0, 3 ($ i = 2 $), resp. 3, 0, 1, 0 ($ i = 3 $) for the outer vertices and also plus $i$ for the labels of the stacked 4-cycle.
                The inter-bag edges are drawn thick, and indeed the labels of their endpoints differ exactly by 1 (mod 4).
            }
            \label{fig:product-C_4}
    \end{figure}
    
    \begin{figure}
        \centering
        \includegraphics{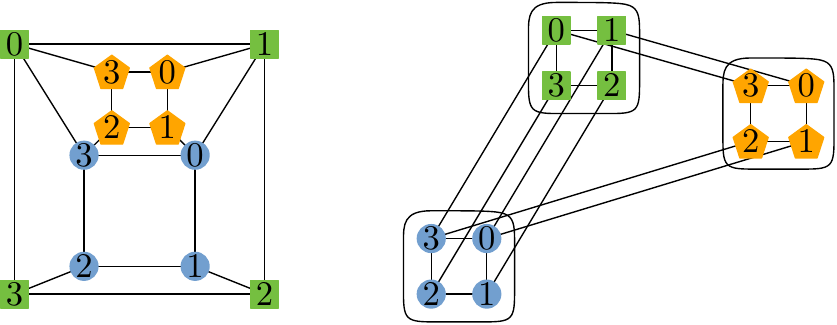}
        \caption{A stacked quadrangulation $ G $ with its $ H $-partition showing $ G \subseteq H \boxtimes C_4 $, where $ H $ is a triangle in this case. The labels used in the proof of \cref{thm:product-stacked-quadrangulations} are written inside the vertices.}
        \label{fig:product_stacked_quadrangulation_example}
    \end{figure}

\Cref{thm:product-stacked-quadrangulations} gives an upper bound of 21 on the queue number of stacked quadrangulations, compared to the upper bound of 5 on the queue number of planar 3-trees (stacked triangulations)~\cite{ABGKP18}.

\qnStackedQuadrangulationsUB*

\begin{proof}
    In general, we have that $\qn(H_1 \boxtimes H_2) \leq |V(H_2)| \cdot \qn(H_1) + \qn(H_2)$ by taking a queue layout of $ H_2 $ and replacing each vertex with a queue layout of $ H_1 $~\cite[Lemma 9]{DJMMUW20}.
    In particular, we conclude that $\qn(H \boxtimes C_4) \leq 4 \cdot \qn(H) + 1$.
    As the queue number of planar $3$-trees is at most 5~\cite{ABGKP18}, \cref{thm:product-stacked-quadrangulations} gives an upper bound of $ 4 \cdot 5 + 1 = 21 $.
 \end{proof}

\section{Lower Bounds}
\label{sec:lower-bounds}

\begin{figure}
    \centering
    \includegraphics{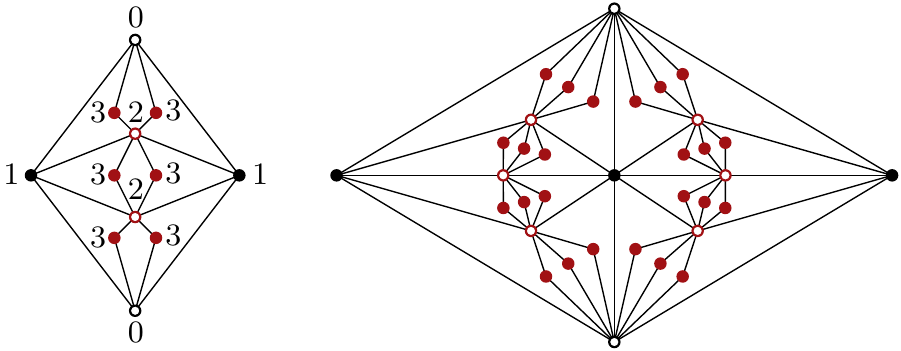}
    \caption{The graphs $ G_3(2) $ and $ G_3(3)$, where the numbers and style of the vertices indicate their depth.}
    \label{fig:G_3(2)}
\end{figure}

This section is devoted to lower bounds on the queue number and mixed page number of bipartite planar graphs.
We use the same family of 2-degenerate quadrangulations $ G_d(w) $ for both lower bounds.
The graph $ G_d(w) $ is defined as follows, where we call $ d $ the \emph{depth} and $ w $ the \emph{width} of $ G_d(w) $; 
see \cref{fig:G_3(2)}.
Let $ G_0(w) $ consist of two vertices, which we call \emph{depth-0 vertices}.
For $ i \geq 0 $, the graph $ G_{i + 1}(w) $ is obtained from $ G_i(w) $ by adding $ w $ vertices into each inner face (except $i=0$, where we use the unique face) and connecting each of them to the two depth-$ i $ vertices of this face (if two exist).
If the face has only one depth-$ i $ vertex $ v $ on the boundary, then we connect the new vertices to $ v $ and the vertex opposite of $ v $ with respect to the face, that is\ the vertex that is not adjacent to $ v $.
The new vertices are then called \emph{depth-$ (i + 1) $ vertices}.
Observe that the resulting graph is indeed a quadrangulation and each inner face is incident to at least one and at most two depth-$ (i + 1) $ vertices.
The two neighbors that a depth-$ (i + 1) $ vertex $ u $ has when it is added are called its \emph{parents}, and $ u $ is called a \emph{child} of its parents.
If two vertices $u$ and $v$ have the same two parents, they are called \emph{siblings}.
We call two vertices of the same depth a \emph{pair} if they have a common child.

\subsection{Queue Layouts}
\label{sec:qn-lb}

We prove combinatorially that $ G_d(w) $ does not admit a 2-queue layout for $d\geq 3,  w \geq 24$. 
We also  verified with a \textsc{sat}-solver~\cite{bob} that the smallest graph in the family of queue number 3 is $ G_4(4) $ containing 259 vertices.


\qnBipartiteLB*

\begin{proof}
   We show that the queue number of $ G_d(w) $ is at least 3 for $ d  \geq 4$  and $ w \geq 24 $. Assume to the contrary that $G_d(w)$ admits a 2-queue layout with vertex order $\prec$. Our goal is to determine some forbidden configurations that will lead to a contradiction. A  \emph{\goodconfig} in a 2-queue layout of $G_d(w)$ consists of six vertices ordered $ x_1 \prec x_2 \prec v \prec v' \prec x_3 \prec x_4 $, where $ v $ and $ v' $ form a pair and $(x_1, x_3, x_4, x_2)$ is a $4$-cycle; see \cref{fig:good-config}.
   A \emph{nested configuration} consists of seven vertices $u\prec u'\prec v_1\prec v_2\prec v_3\prec v_4\prec v_5$ , where $u$ and $u'$ form a pair, and contains edges $(u,v_2)$, $(u,v_4)$, $(u',v_1)$ and $(u',v_4)$; see  \cref{fig:nested-config}.
   The \emph{depth} of a configuration is defined as the depth of the pair $ \langle v, v' \rangle $, respectively $ \langle u, u' \rangle $.
   
   \begin{figure}
       \centering%
       \begin{subfigure}[b]{0.4\textwidth}
       \includegraphics[page=1]{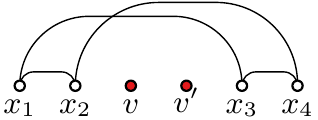}
       \caption{\nolinenumbers}
       \label{fig:good-config}
       \end{subfigure}
       \hfil
        \begin{subfigure}[b]{0.48\textwidth}
       \includegraphics[page=4]{linear_layout_figures}
       \caption{\nolinenumbers}
       \label{fig:nested-config}
       \end{subfigure}
       \caption{ (a)~A \goodconfig with pair $\langle v, v' \rangle$. (b) A nested configuration with pair $\langle u, u' \rangle$.}
       \label{fig:configs}
   \end{figure}
   
   \begin{figure}
       \centering%
       \begin{subfigure}[b]{0.4\textwidth}
       \includegraphics[page=2]{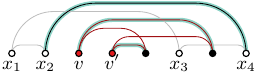}
       \caption{\nolinenumbers}
       \label{fig:good-right-1}
       \end{subfigure}
       \hfil
       \begin{subfigure}[b]{0.48\textwidth}
		\centering
       \includegraphics[page=3]{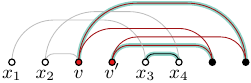}
       \caption{\nolinenumbers}
       \label{fig:good-right-2}
       \end{subfigure}
       
        \begin{subfigure}[b]{0.4\textwidth}
       \includegraphics[page=24]{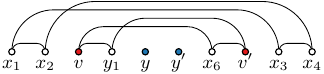}
       \caption{\nolinenumbers}
       \label{fig:good-middle}
       \end{subfigure}
       \hfil
       \begin{subfigure}[b]{0.48\textwidth}
       \includegraphics[page=25]{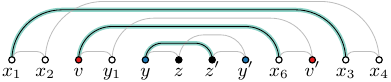}
       \caption{\nolinenumbers}
       \label{fig:good-middle_2}
       \end{subfigure}
       \caption{%
           Proof of Claim~\ref{claim:no_good_config}.
           (a)--(b)~A \goodconfig with three children to the right of $ v' $.
           Either two of them are to the left of $ x_4 $ as in (a) or to the right of $x_4$ as in (b). Both cases create a 3-rainbow.
           (c)~If at least 10 children are between $v$ and $v'$, there is another \goodconfig formed by $v, v'$ and their children. (d)~Three nested {\goodconfig}s contain a $3$-rainbow.
           }
       \label{fig:good-right}
   \end{figure}

   \begin{claim}\label{claim:no_good_config}
       For $w\geq 14$, a $2$-queue layout of $G_d(w)$ does not contain a \goodconfig at depth $d'\leq d-2$.
   \end{claim}
   \begin{claimproof}
       Consider a bad configuration and 14 children of $ v $ and $ v'$, partitioned into seven pairs. 
       By pigeonhole principle, we have at least ten children between $ v $ and $ v' $ or we have at least three children either to the left of $v$ or to the right  of $v'$. 
       
       Assume first that we have at least three children of $\langle v,v'\rangle$ to the right of $v'$ (note that the case where at least three children are to the left of $v$ is symmetric). Clearly, we have  two children between $ v' $ and $ x_4 $ or two of them to the right of $ x_4 $ (see \cref{fig:good-right}). 
       In both cases, the edges connecting them to $v$ and $v'$ form a 2-rainbow that is nested by $(x_2,x_4)$ or nests  $(x_3,x_4)$, as shown in Figures~\ref{fig:good-right-1} and \ref{fig:good-right-2} respectively.
       So we have that a 3-rainbow cannot be avoided if there are three children of $ v $ and $ v' $ to the right of $ v' $ or, by symmetry, to the left of $ v $.
       
       Second, assume that at least ten children of $\langle v,v'\rangle$ are placed between $ v $ and $ v' $, while at most four of them are not.
       We aim to find another \goodconfig consisting of $ v, v' $, and four of their children.
       Since we have a total of 14 children forming seven pairs, and at most four children are not between $v$ and $v'$, it follows that at least three pairs of children are placed between $v$ and $v'$. Let $y_1$ be the first child to the right of $v$ and $y_6$ the first child to the left of $v'$. Then, at least one pair of children, say $\langle y,y' \rangle$, where $\{y_1,y_6\}\cap \{y,y'\}=\emptyset$ is placed between $y_1$ and $y_6$. Note that vertices $v,y_1,y,y',y_6,v'$ form a new \goodconfig at depth $d'+1$, whose edges are nested by the edges of the starting \goodconfig; the situation is depicted in \cref{fig:good-middle}.
       
       We repeat the process and consider the children of $\langle y,y'\rangle$. If at least three of them are either to the left of $y$ or the right of $y'$, then there exists a $3$-rainbow. Otherwise, there exists a pair of children $\langle z, z'\rangle$ of $y$ and $y'$ (at depth $d'+2$) that are between $y$ and $y'$. It is not hard to see that  edges $(x_1,x_3)$, $(v,x_6)$ and $(y,z')$ create a $3$-rainbow; see \cref{fig:good-middle_2}.
    \end{claimproof}

   \begin{figure}
       \centering%
        \begin{subfigure}[b]{0.45\textwidth}
       \includegraphics[page=5]{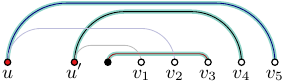}
       \caption{\nolinenumbers}
       \label{fig:v_3-left-children_1}
       \end{subfigure}
       \hfil
       \begin{subfigure}[b]{0.45\textwidth}
       \includegraphics[page=6]{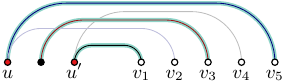}
       \caption{\nolinenumbers}
       \label{fig:v_3-left-children_2}
       \end{subfigure}
       \hfil
       \begin{subfigure}[b]{0.45\textwidth}
       \includegraphics[page=7]{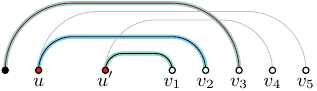}
       \caption{\nolinenumbers}
       \label{fig:v_3-left-children_3}
       \end{subfigure}
       \hfil
       \begin{subfigure}[b]{0.45\textwidth}
       \includegraphics[page=26]{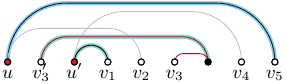}
       \caption{\nolinenumbers}
       \label{fig:v_3-left-children_pair_1}
       \end{subfigure}
       \hfil
       \begin{subfigure}[b]{0.45\textwidth}
       \includegraphics[page=27]{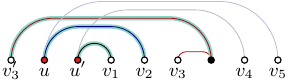}
       \caption{\nolinenumbers}
       \label{fig:v_3-left-children_pair_2}
       \end{subfigure}
       \caption{Proof of \cref{claim:nested_config}. (a)--(c)~Any edge leaving $ v_3 $ to the left creates a 3-rainbow. (d)--(e)~The pair of $v_3$ does not precede $u'$.}
       \label{fig:v_3-left-children}
   \end{figure}

   \begin{claim}\label{claim:nested_config}
       If a 2-queue layout of $G_d(w)$ contains a nested configuration at depth $d'\leq d-1$, then  $v_3$ precedes all its children. Additionally, if $v_3$ forms a pair with $v'_3$, then $v'_3$ is to the right of $u'$.
   \end{claim}
   \begin{claimproof}
   Consider a nested configuration at depth $d'<d$. Then $v_3$ has $w$ children. If one of them is to the left of $v_3$, then it is either between $u'$ and $v_3$, or between $u$ and $u'$ or to the left of $u$. In all three cases there is a $3$-rainbow; see \cref{fig:v_3-left-children_1,fig:v_3-left-children_2,fig:v_3-left-children_3}. For the second part of the lemma, assume that $v'_3$ precedes $u'$. Then it is either between $u$ and $u'$ or before $u$. Again a $3$-rainbow is created; see Figures~\ref{fig:v_3-left-children_pair_1} and \ref{fig:v_3-left-children_pair_2}.
    \end{claimproof}
   
   Consider the initial pair $ \langle u, u'\rangle $ of $ G_0(w) $ with 24 children, grouped into twelve pairs. Without loss of generality assume that $u\prec u'$ in a 2-layout of $G_d(w)$. If there are at least three pairs between $u$ and $u'$, then they form a \goodconfig with $u$ and $u'$ at depth~$0$, contradicting \cref{claim:no_good_config}. So, there are at most two pairs between $ u $ and $ u' $, and at least ten pairs have (at least) one vertex either to the left of $u$ or the right of $u'$.   Hence we can assume without loss of generality that at least five vertices of five different pairs are to the right of $u'$. Let $v_1$, $v_2$, $v_3$, $v_4$ and $v_5$ be these five vertices in the order they appear after $u'$. We assume without loss of generality that $v_1$, $v_2$, $v_3$, $v_4$ and $v_5$ are the rightmost possible choices, in particular, this means that for the pair $\langle v_i,v_i'\rangle$ it holds that $v_i' \prec v_i$, for $1\leq i\leq 5$. As they are children of $\langle u,u'\rangle$, they form a nested configuration. By \cref{claim:nested_config} and the fact that $v'_3\prec v_3$, we conclude that $v'_3$ is between $u'$ and $v_3$, while the children of $\langle v_3,v'_3\rangle$ are to the right of $v_3$. 
   

   \begin{figure}
       \centering%
       \includegraphics{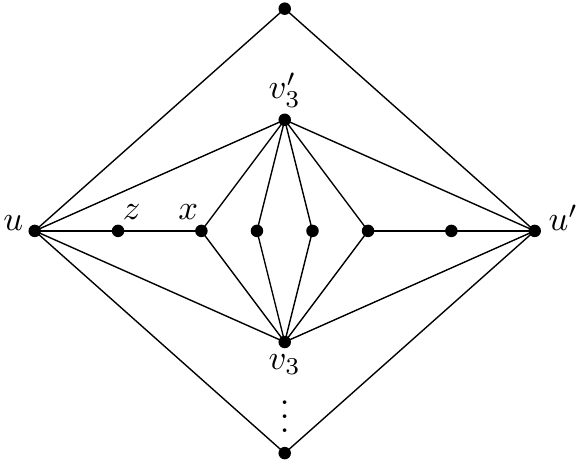}
       \caption{Parts of the graph whose queue number is at least 3 constructed for \cref{thm:qn-bipartite-lb}}
       \label{fig:lb-planar}
   \end{figure}
   
   \begin{figure}
       \centering%
       \begin{subfigure}[b]{0.25\textwidth}
       \includegraphics[page=8]{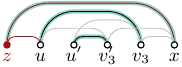}
       \caption{\nolinenumbers}
       \label{fig:inserting-z-1}
       \end{subfigure}
       \hfil
       \begin{subfigure}[b]{0.33\textwidth}
       \includegraphics[page=9]{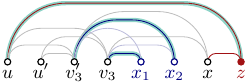}
       \caption{\nolinenumbers}
       \label{fig:inserting-z-2}
       \end{subfigure}
       \hfil
       \begin{subfigure}[b]{0.33\textwidth}
       \includegraphics[page=10]{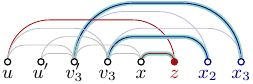}
       \caption{\nolinenumbers}
       \label{fig:inserting-z-3}
       \end{subfigure}
       \caption{Three cases that can occur when inserting $ z $ into the layout without creating a \goodconfig. In all three cases we have a 3-rainbow.}
       \label{fig:inserting-z}
   \end{figure}

   Now consider the children of $\langle  v_3 ,v'_3 \rangle $ and let $ x $ denote the child that shares a face with $ u $. 
   Denote by $x_i$, $1\leq i\leq 3$,
   any three children different from $x$  such that $x_1 \prec x_2 \prec x_3$; see \cref{fig:lb-planar}. By \cref{claim:nested_config}, all children of $\langle  v_3 ,v'_3 \rangle $ are to the right of $v_3$. In particular, either $x_1\prec x_2\prec x$ or $x\prec x_2\prec x_3$. As $u$ and $x$ belong to the boundary of a face, and $x$ is at depth $2<d$, $u$ and $x$ have $w$ common children. By \cref{claim:no_good_config}, there exists a pair of children $\langle z, z'\rangle $ of $\langle u, x \rangle$ such that $z$ is not located between $u$ and $x$. Hence there are two cases to consider, namely $z$ is to the left of $u$, or $z$ is to the right of $x$. In the first case, edges $(u',v_3')$, $(u,v_3)$ and $(x,z)$ form a $3$-rainbow; see \cref{fig:inserting-z-1}. So, $z$ is to the right of $x$. If $x_1\prec x_2\prec x$, then we have the situation depicted in \cref{fig:inserting-z-2}, otherwise $x\prec x_2\prec x_3$  as in \cref{fig:inserting-z-3}. In both cases a $3$-rainbow is created.
   We conclude that $G_d(w)$ with $d=4$ and $w=24$ does not admit a 2-queue layout.
 \end{proof}

\subsection{Mixed Linear Layouts}
\label{sec:mixed-lb}

Next, we prove that for $d \geq 3$ and $w \geq 154$, graph $ G_d(w) $ does not admit a 1-queue 1-stack layout. We remark that the smallest graph of this family with this property is actually $G_3(5)$ which has $128$ vertices. Again, we verified this with a \textsc{sat}-solver~\cite{bob}.
Note that the following theorem answers a question raised in \cite{Pup18,CKN19,ABKM20}.


\mixedLB*

\begin{proof}
   Assume for the sake of a contradiction that $G_3(154)$ admits a 1-queue 1-stack layout with vertex order $ \prec $ and let $ u, u' $ denote the two initial vertices with $ u \prec u' $.
   For ease of presentation, we call an edge \emph{blue} (\emph{orange}) if it is in the queue (stack) and color the edges accordingly in all figures.
   We distinguish three types of children:
   A child $ x $ with parent pair $ \langle v, v' \rangle $ is called a \emph{blue (orange) child} if both edges $ (v,x) $ and $ (v',x) $ are blue (orange, resp.), and it is called \emph{bicolored} if one of the edges is blue and the other is orange.
   A pair is called \emph{blue} (\emph{orange}) if both vertices are blue (orange) and \emph{bicolored} if it contains a bicolored vertex.
   
   \begin{figure}
       \centering%
       \begin{subfigure}[b]{0.22\textwidth}
       \centering
       \includegraphics[page=11]{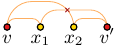}
       \caption{\nolinenumbers}
       \label{fig:yellow-child-1}
       \end{subfigure}
       \hfil
       \begin{subfigure}[b]{0.22\textwidth}
       \centering
       \includegraphics[page=12]{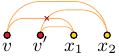}
       \caption{\nolinenumbers}
       \label{fig:yellow-child-2}
       \end{subfigure}
       \hfil
       \begin{subfigure}[b]{0.22\textwidth}
       \centering
       \includegraphics[page=13]{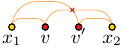}
       \caption{\nolinenumbers}
       \label{fig:yellow-child-3}
       \end{subfigure}
       \hfil
       \begin{subfigure}[b]{0.22\textwidth}
       \centering
       \includegraphics[page=14]{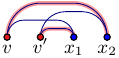}
       \caption{\nolinenumbers}
       \label{fig:blue-child}
       \end{subfigure}
       \caption{Children of a pair $\langle v,v' \rangle$ with $v\prec v'$. (a)~There is at most one orange child between $v$ and $v'$. (b)--(c)~There is at most one orange child to the left of $v$ or to the right of $v'$. (d)~There is at most one blue child to the right of $v'$.}
       \label{fig:yellow-child}
   \end{figure}
   
   Consider a pair $ \langle v, v' \rangle $ with $v \prec v'$. We first make two preliminary observations.
   \begin{claim}\label{claim:orange}
   A pair $ \langle v,v' \rangle $ with $ v \prec v' $ has at most two orange children, one between $ v $ and $ v'$ and one to the left of $v$ or to the right of $v'$.
   \end{claim}
   \begin{proof}
       Suppose first that $ \langle v,v' \rangle $  has two orange children $x_1,x_2$ such that $v  \prec x_1 \prec x_2 \prec v'$. Then, edges $(v,x_2)$ and $(v',x_1)$ cross (\cref{fig:yellow-child-1}), a contradiction. 
       %
       Second, consider the case $v  \prec v' \prec x_1 \prec x_2$ (the case $x_1 \prec x_2 \prec v \prec v'$ is symmetric). Here $(v,x_1)$ and $(v',x_2)$ cross (\cref{fig:yellow-child-2}). Finally, if $x_1 \prec v \prec v' \prec x_2$, then $(x_1,v')$ and $(x_2,v)$ cross (\cref{fig:yellow-child-3}).
    \end{proof}
   
   \begin{claim}\label{claim:blue}
   A pair $ \langle v,v' \rangle $  has at most two blue children that are not located between $v$ and $v'$, namely one to the left of $v$ and one to the right of $v'$.
   \end{claim}
   \begin{claimproof}
   Assume for a contradiction that $ \langle v,v' \rangle $ has two blue children $x_1,x_2$ to the right of $v'$. Then, edges $(v,x_2)$ and $(v',x_1)$ nest (\cref{fig:blue-child}); a contradiction.
    \end{claimproof}
   
   We group the children of every pair $\langle v, v' \rangle$ at depth $d'<d$ into 77 pairs each. Then, we ignore any pair containing a blue vertex that is not placed between its parents or containing a orange vertex (no matter where it is placed). That is, by Claims~\ref{claim:orange} and~\ref{claim:blue} we discard at most four pairs of children for each pair $\langle v, v' \rangle$, that is at least 146 children out of 154 that are grouped into 73 pairs remain. Let $G'$ denote the resulting subgraph of $G_d(w)$. By definition of subgraph $G'$, the following property holds:
   
   \begin{claim}\label{claim:edgesInGPrime}
   Let $\langle v, v' \rangle$ denote a pair occurring in $G'$. Then \begin{inlineenum}
   \item\label{claim:edgesInGPrime:1} $\langle v, v' \rangle$   is either bicolored or blue in $G'$ and
   \item\label{claim:edgesInGPrime:2} all blue children of $\langle v, v' \rangle$ in $G'$ are placed between $v$ and $v'$ in $\prec$. 
   \end{inlineenum}
   \end{claim}
   %
   %
     \begin{figure}
       \centering%
       \begin{subfigure}[b]{0.45\textwidth}
       \centering
       \includegraphics[page=30]{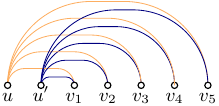}
       \caption{\nolinenumbers}
       \label{fig:mixed-configuration-1}
       \end{subfigure}
       \hfil
       \begin{subfigure}[b]{0.45\textwidth}
       \centering
       \includegraphics[page=31]{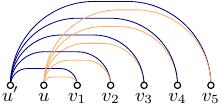}
       \caption{\nolinenumbers}
       \label{fig:mixed-configuration-2}
       \end{subfigure}
       \caption{The two different possible layouts of vertices $u,u',v_1,\ldots,v_5$ of a \mixedconfig.}
       \label{fig:mixed-configuration}
   \end{figure}
   We now consider the 146 children of a pair $\langle u,u' \rangle$ at depth $d'\leq d-2$ in the linear layout of $G'$ induced by the linear layout of $G$, grouped into 73 pairs. Consider the following configuration. The pair $\langle u,u'\rangle$ has five pairs $\langle v_i,v'_i\rangle$ (for $1 \leq i \leq 5$) of children and pair $\langle v_i,v'_i\rangle$ has a child $z_i$. We call this is a  \emph{\mixedconfig} if \begin{inlineenum}
   \item $u,u',v'_i\prec v_i$, \item $(u,v_i)$ is orange while $(u',v_i)$ is blue,
   \item all edges $(v_i,z_i)$ (for $1 \leq i \leq 5$) have the same color;
   \end{inlineenum} see \cref{fig:mixed-configuration}.
   
   \begin{claim}\label{claim:mixed_config_blue}
        In the mixed layout of $G'$, there is no \mixedconfig with edges $(v_i,z_i)$ (for $1 \leq i \leq 5$) being blue.
   \end{claim}
   \begin{claimproof}
        \begin{figure}
            \centering%
                \begin{subfigure}[b]{0.45\textwidth}
            \centering
            \includegraphics[page=32]{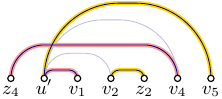}
            \caption{\nolinenumbers}
            \label{fig:right-blue-0}
            \end{subfigure}
            \hfil
            \begin{subfigure}[b]{0.45\textwidth}
            \centering
            \includegraphics[page=33]{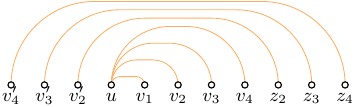}
            \caption{\nolinenumbers}
            \label{fig:right-blue-01}
            \end{subfigure}
            \hfil
            \begin{subfigure}[b]{0.45\textwidth}
            \centering
            \includegraphics[page=21]{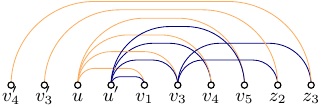}
            \caption{\nolinenumbers}
            \label{fig:right-blue-1}
            \end{subfigure}
            \hfil
            \begin{subfigure}[b]{0.45\textwidth}
            \centering
            \includegraphics[page=22]{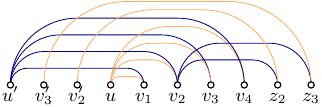}
            \caption{\nolinenumbers}
            \label{fig:right-blue-2}
            \end{subfigure}
            \caption{The edges $ v_i z_i $ are blue. Connecting $ v'_3 $ to its parents $ u $ and $ u' $ yields a contradiction.}
            \label{fig:right-blue}
        \end{figure}
    
        Observe that, for $2 \leq i \leq 4$, $ v_5 \prec z_i $ as otherwise $(v_i,z_i)$ nests $(u',v_1)$ or is nested by $(u',v_5)$; see \cref{fig:right-blue-0}.
        Thus by Claim~\ref{claim:edgesInGPrime}\eqref{claim:edgesInGPrime:2}, vertices $z_2, z_3, z_4 $ are bicolored and therefore edges $(v'_2 z_2)$, $(v'_3 z_3)$ and $ (v'_4,z_4) $ are orange.
        As orange edges may not cross, $ v'_2, v'_3, v'_4 $ are to the left of $ u $ (we already have $ v'_i \prec v_i $ by the definition of a \mixedconfig). In particular $ v'_4 \prec v'_3 \prec v'_2 \prec u $; see \cref{fig:right-blue-01}.
        Note that we do not know the position of $ u' $ in the vertex order so far. First assume that $v'_3 \prec u'$; see \cref{fig:right-blue-1}, where $u$ is drawn to the left of $u'$ as in~\cref{fig:mixed-configuration-1} (note that the following argument also applies if $u' \prec u$ as in~\cref{fig:mixed-configuration-2}). As $v'_4\prec v'_3$, vertex $ v'_4 $ is to the left of both $u$ and $u'$ and therefore bicolored, by Claim~\ref{claim:edgesInGPrime}\eqref{claim:edgesInGPrime:2}. However, the edges $(v'_4,u)$ and $(v'_4,u')$ both cross the orange edge $(v'_3,z_3)$. Thus, we have $u' \prec v'_3$; see \cref{fig:right-blue-2}. Here it holds that $ u' \prec v'_3 \prec v'_2 \prec u \prec v_1 $. In this case, edge $(v'_3,u)$ is nested by the blue edge $(u',v_1)$, hence, it cannot be blue. On the other hand, it crosses the orange edge $(v'_2,z_2)$, so it cannot be orange either.
    \end{claimproof}

    \begin{claim}\label{claim:mixed_config_orange}
        In the mixed layout of $G'$, there is no \mixedconfig with edges $(v_i,z_i)$ (for $1 \leq i \leq 5$) being orange.
    \end{claim}
    \begin{claimproof}
        \begin{figure}
            \centering%
            \begin{subfigure}[b]{0.4\textwidth}
            \centering
            \includegraphics[page=34]{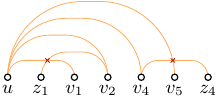}
            \caption{\nolinenumbers}
            \label{fig:right-yellow-1}
            \end{subfigure}
            \hfil
            \begin{subfigure}[b]{0.55\textwidth}
            \centering
            \includegraphics[page=23]{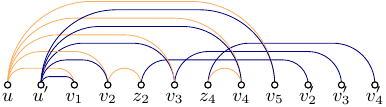}
            \caption{\nolinenumbers}
            \label{fig:right-yellow-2}
            \end{subfigure}
            \caption{The edges $ y_i c_i $ are orange. Connecting $ v'_4 $ to its parents $ u $ and $ u' $ yields a contradiction.}
            \label{fig:right-yellow}
        \end{figure}

        Consider edge $(v_i,z_i)$ for $2\leq i\leq 4$. If $z_i$ is to the left of $v_{i-1}$ then it crosses the orange edge $(u,v_{i-1})$, and if it is to the right of $v_{i+1}$ then it crosses the orange edge $(u, v_{i+1})$; see \cref{fig:right-yellow-1}. So $ v_1 \prec z_2 \prec z_4 \prec v_5 $ holds.
        Recall that $G'$ contains no orange children so for $1 \leq i\leq 5$, the edge $(v'_i z_i )$ is blue.
        Since $(u',v_1)$ is blue, for $2 \leq i \leq 5$,  $v'_i$ cannot precede $u'$, as otherwise $(v'_i,z_i)$ would nest $(u',v_1)$. Similarly, since $(u',v_5)$ is blue, for $1 \leq i \leq 4$, vertex $v'_i$ cannot be between $u'$ and $v_5$, as otherwise $(v'_i,z_i)$ would be nested by $(u',v_5)$. Thus, we conclude that $ v_5 \prec v'_2 \prec v'_3 \prec v'_4 $; see \cref{fig:right-yellow-2}.
        Since $ v'_4 $ is to the right of $ u$ and $ u'$, by  Claim~\ref{claim:edgesInGPrime}, it must be bicolored, 
        that is, either edge $(u,v_4')$ or edge $(u',v_4')$ is blue. However, both these edges nest the blue edge $(v'_2, z_2)$, a contradiction.
    \end{claimproof}

   Combining Claims~\ref{claim:mixed_config_blue} and~\ref{claim:mixed_config_orange}, we conclude that the linear layout of $G'$ contains no \mixedconfig. This property allows us to conclude the following:

   \begin{claim}\label{claim:atLeastFivePairs}
        Let $\langle u, u' \rangle$ be a pair in $G'$  at depth $d' \leq d-2$. In the linear layout of $G'$ there exist at least five pairs of children of $\langle u, u' \rangle$ with both vertices of each pair between $u$ and $u'$.
   \end{claim}
   
   \begin{claimproof} 
        Assume for the sake of contradiction that $\langle u, u' \rangle$  have  at most four such pairs of children. Thus, for at least 69 pairs, at least one vertex is not located between $u$ and $u'$. In particular, at least 35 of these vertices either precede both $u$ and $u'$, or are placed to the right of both $u$ and $u'$. Assume without loss of generality that the latter applies.
        Recall that children to the right of both parents are bicolored by Claim~\ref{claim:edgesInGPrime}\eqref{claim:edgesInGPrime:2}.
        Now, among these 35 bicolored children, at least 18 are connected to the same parent $u$ or $u'$ with an orange edge; 
        without loss of generality to $u$. 
        These 18 children belong to at least 9 pairs, so we can select nine of them that do not form a pair, and in particular we select the nine rightmost ones, say $v_1$ to $v_9$.
        Observe that $v'_i\prec v_i$, for $1\leq i\leq 9$. Indeed, if $v_i'$ is bicolored and $v_i\prec v'_i$,  then $v_i'$ would have been chosen instead of $v_i$. On the other hand, if $v_i'$ is blue, then it is located between $u$ and $u'$ by  Claim~\ref{claim:edgesInGPrime}\eqref{claim:edgesInGPrime:2} and thus $v_i' \prec v_i$. 
        Now let $ z_i $ be a child of $ \langle v_i, v'_i\rangle $ ($1\leq i\leq 9$).
        By the pigeonhole principle, either five of the edges $(v_i,z_i)$ with $1 \leq i \leq 9$ are blue or five of these edges are orange. Hence a \mixedconfig is formed, contradicting 
        Claims~\ref{claim:mixed_config_blue} and~\ref{claim:mixed_config_orange}.
        %
        %
        %
        %
    \end{claimproof}

   Since $d > 2$, we conclude that there are pairs $\langle u,u' \rangle$ which have at least five pairs of children for which both vertices are located between $u$ and $u'$. We now  investigate this case.
    %
     
   %
   \begin{claim}\label{claim:noPairsAtGrandchildren}
   Let $\langle u,u'\rangle$ be a pair at depth $d'\leq d-2$. In the linear layout of $G'$ there is a pair $\langle v, v' \rangle$ of children of $\langle u,u'\rangle$, so that no child of $\langle v, v' \rangle$ is located between $v$ and $v'$.
   \end{claim}
  
   \begin{claimproof}
   By Claim~\ref{claim:atLeastFivePairs}, the pair $\langle u,u'\rangle$ has five pairs of children $\langle v_1,v'_1 \rangle, \ldots, \langle v_5,v'_5 \rangle$  for which both vertices are located between $u$ and $u'$. Assume without loss of generality that $u\prec u'$. As there are no orange pairs in $G'$, either three of them are blue or three of them are bicolored. 
     \begin{figure}
       \centering%
       \begin{subfigure}[b]{0.3\textwidth}
       \centering
       \includegraphics[page=16]{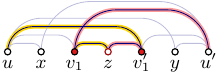}
       \caption{\nolinenumbers}
       \label{fig:middle-blue-1}
       \end{subfigure}
       \hfil
       \begin{subfigure}[b]{0.3\textwidth}
       \centering
       \includegraphics[page=17]{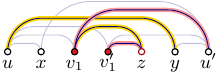}
       \caption{\nolinenumbers}
       \label{fig:middle-blue-2}
       \end{subfigure}

       \begin{subfigure}[b]{0.3\textwidth}
       \centering
       \includegraphics[page=18]{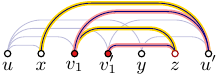}
       \caption{\nolinenumbers}
       \label{fig:middle-blue-3}
       \end{subfigure}
       \hfil
       \begin{subfigure}[b]{0.3\textwidth}
       \centering
       \includegraphics[page=19]{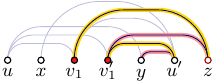}
       \caption{\nolinenumbers}
       \label{fig:middle-blue-4}
       \end{subfigure}
       \caption{If there are three blue pairs between $u$ and $u'$, one of them, namely $\langle v_1,v'_1 \rangle$, is between a vertex $x$ and a vertex $y$ belonging to other blue pairs. A child $z$ of $\langle v_1,v'_1 \rangle$ results in a nesting.}
       \label{fig:middle-blue}
   \end{figure}
   Assume first that three of these pairs, say $\langle v_1,v'_1 \rangle$, $\langle v_2,v'_2 \rangle$, $\langle v_3,v'_3 \rangle$ are blue. Then there is a blue pair among them, say without loss of generality $ \langle v_1, v'_1 \rangle $, such that $ u \prec x \prec v_1 \prec v'_1 \prec y \prec u' $ where $x,y \in \{v_2,v'_2,v_3,v'_3\}$.
   Now consider any child $ z $ of $ v_1 $ and~$ v'_1$. Since by definition $G'$ contains no orange children, $z$ is connected with a blue edge to either $v_1$ or $v'_1$. Vertex $z$ may be placed either to the left of $v_1$, or between $v_1$ and $v_1'$ or to the right of $v_1'$. If it is between $v_1$ and $v_1'$, the edge $(v_1,z)$ is nested by $(u,v_1')$ and the edge $(z,v_1')$ is nested by $(v_1,u')$;~
   see \cref{fig:middle-blue-1}. Thus $z$ cannot be connected to $v_1$ or $v_1'$ with a blue edge; a contradiction. 
   Assume without loss of generality that $z$ is to the right of $v_1'$. 
   It is easy to verify that for all possible placements of $z$, the blue edge connecting $z$ to $v_1$ or $v_1'$ either nests or is nested by an edge incident to one of $u$ and $u'$; see \cref{fig:middle-blue-2,fig:middle-blue-3,fig:middle-blue-4}. 
   
   \begin{figure}
       \centering%
       \begin{subfigure}[b]
       {0.3\textwidth}
       \centering
       \includegraphics[page=29]{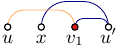}
       \caption{\nolinenumbers}
       \label{fig:middle-bicolored_1}
       \end{subfigure}
       \hfil
       \begin{subfigure}[b]
       {0.3\textwidth}
       \centering
       \includegraphics[page=28]{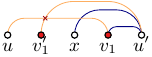}
       \caption{\nolinenumbers}
       \label{fig:middle-bicolored_2}
       \end{subfigure}
       \hfil
       \begin{subfigure}[b]{0.3\textwidth}
       \centering
       \includegraphics[page=20]{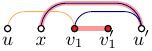}
       \caption{\nolinenumbers}
       \label{fig:middle-bicolored_3}
       \end{subfigure}
       \caption{Three bicolored pairs $\langle v_i,v'_i\rangle$ of children of $\langle u, u'\rangle$ are between $u$ and $u'$. 
       	(a)~Without loss of generality vertex $v_1$ is nested by a blue edge. (b)~Vertex $v'_1$ does not precede $x$. (c)~Blue and bicolored children of $ \langle v_1, v'_1 \rangle $ are to the left or to the right of both parents.}
       \label{fig:middle-bicolored}
   \end{figure}

  Thus, there are three bicolored  pairs between $u$ and $u'$, say $\langle v_1,v'_1 \rangle$, $\langle v_2,v'_2 \rangle$, $\langle v_3,v'_3 \rangle$, such that $ v $ is a bicolored vertex while $ v' $ can be  blue or bicolored.
  Without loss of generality two of $ v_1, v_2, v_3 $, say $v_1$ and $v_2$ are connected to $ u $ with an orange edge. 
  Let $x$ be the leftmost among $v_1,v'_1,v_2,v'_2,v_3,v'_3$ in $\prec$ that is connected to $u'$ with a blue edge. 
  As a result for one of $v_1$ and $v_2$, say without loss of generality $v_1$, we have $v'_1 \neq x$ and $u \prec x \prec v_1 \prec u'$; see \cref{fig:middle-bicolored_1}.
  %
  %
  Observe that $ x \prec v'_1 $. Otherwise $(v'_1,u')$ cannot be blue by the choice of $ x $ and cannot be orange either, as it would cross the orange edge $(u, v_1)$; see \cref{fig:middle-bicolored_2}.
  Now the blue edge $(x, u')$ nests above $ v_1 $ and $ v'_1 $ and thus blue and bicolored children of $\langle v_1, v_1' \rangle $ cannot be between them; see \cref{fig:middle-bicolored_3}.  
  \end{claimproof}
  
  We are now ready to prove the theorem. 
  Namely, we have the initial pair $\langle u, u' \rangle$ of $G_0(w)$ at depth $0 \leq d-2$. Then, by Claim~\ref{claim:atLeastFivePairs} there is a pair $\langle v, v' \rangle$ at depth $1$ such that no child of $\langle v, v' \rangle$ is located between $v$ and $v'$.
  Since however $\langle v, v' \rangle$ is at depth $1 \leq d-2$, by  Claim~\ref{claim:atLeastFivePairs}, at least $10$ of the children of $\langle v, v' \rangle$ must be between $v$ and $v'$ in $\prec$; a contradiction.
 \end{proof}

In contrast to our result on the  queue number of bipartite planar graphs (\cref{thm:qn-bipartite-lb}), bipartite planar graphs admit $2$-stack layouts. Therefore, if we increase the number of stacks, we can easily construct a mixed linear layout of $G_d(w)$ (or of any bipartite planar graph). On the other hand, it remains open how many queues are needed if we allow at most one stack.
In the next section we approach this question by showing that the graph $G_d(w)$ constructed above (and even more generally any 2-degenerate quadrangulation) admits a $5$-queue layout.

\section{2-Degenerate Quadrangulations}
\label{sec:2-degenerate}

Note that the graph $G_d(w)$ defined in Section~\ref{sec:lower-bounds} is a 2-degenerate quadrangulation. Recall that it can be constructed from a 4-cycle by repeatedly adding a degree-2 vertex and keeping all faces of length $4$.
Hence, every 2-degenerate quadrangulation is a subgraph of a $4$-tree. 
This can also be observed by seeing $ G_d(w) $ as a $ (4, 1) $-stack graph, together with \cref{thm:product_fs-stacked}.
Thus, by the result of Wiechert~\cite{Wie17}, it admits a layout on $2^4-1 = 15$ queues. 
In this section, we improve this bound by showing that 2-degenerate quadrangulations admit $5$-queue layouts.


Our proof is constructive and uses a special type of tree-partition.
Let $T$ be a tree-partition of a given graph $G$, that is, an $H$-partition where $H$ is a rooted tree.
For every node $x$ of $T$, if $y$ is the parent node of $x$ in $T$, the set of
vertices in $T_y$ having a neighbor in $T_x$ is called the \df{shadow} of $x$; we say that the shadow is \df{contained} in node $y$.
The \df{shadow width} of a tree-partition is the maximum size of a shadow contained in a node of $T$; see \cref{fig:defs}.

\begin{figure}
	\centering
	\begin{subfigure}[b]{0.49\linewidth}
		\centering
		\includegraphics[page=1]{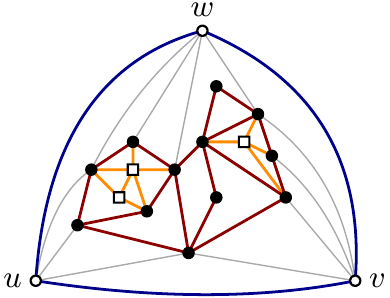}
	\end{subfigure}
	\begin{subfigure}[b]{0.49\linewidth}
		\centering
		\includegraphics[page=2]{tree_partition_concrete_example}
	\end{subfigure}
	\caption{A planar graph and its tree-partition of shadow width $4$.}
	\label{fig:defs}
\end{figure}

Let $\mathcal{S}=\{C_i \subseteq V~:~1\le i \le |\mathcal{S}|\}$ be a collection of vertex subsets for a graph $G=(V, E)$, 
and let $\pi$ be an order of $V$. Consider two elements from $\mathcal{S}$, 
$C_x=[x_1, x_2, \dots, x_{|C_x|}]$ and $C_y=[y_1, y_2, \dots, y_{|C_y|}]$, where the vertices are ordered according 
to $\pi$. We say that $C_x$ \df{precedes} $C_y$ with respect to $\pi$ if
$x_i \le y_i$, for all $1\leq i\leq \min(|C_x|, |C_y|)$;
we denote this relation by $C_1 \prec C_2$.
We say that $\mathcal{S}$ is \df{nicely ordered} if $\prec$ is a total
order on $\mathcal{S}$, that is, $C_i \prec C_j$ for all $1 \le i < j \le |S|$. 
A similar concept of clique orders has been considered in \cite{DMW05} and \cite{Pup20}.


\begin{restatable}{lemma}{mainlem}
	\label{lm:main}
	Let $G=(V, E)$ be a graph with a tree-partition $\big(T, \{T_x : x\in V(T)\}\big)$ of
	shadow width $k$.
	Assume that for every node $x$ of $T$, the following holds:
	\begin{inlineenum}
		\item there exists a $q$-queue layout of $T_x$ with vertex order $\pi$, and
		\item all the shadows contained in $x$ are nicely ordered with respect to $\pi$.
	\end{inlineenum}	
	Then $\qn(G) \le q + k$.
\end{restatable}	

\begin{proof}
In order
	to construct a desired queue layout of $G$, we first build a $1$-queue layout of the nodes of $T$.
	This is done by lexicographic breadth-first search (starting from the root of $T$)
	in which the nodes sharing the same parent are ordered with respect to the given nice order of shadows.	
	Then every node, $x$, in the layout is replaced by the vertices of bag $T_x$; the vertices
	within a bag are ordered with respect to the vertex order $\pi$ of the $q$-queue layout that is guaranteed by the lemma.
	This results in an order of the vertices of $G$ in which every vertex $v \in V$  
	is associated with a pair $\langle i_v, j_v\rangle$, where $i_v$ is derived from the $1$-queue layout of $T$ 
	and $j_v$ is derived from $\pi$. Clearly, the vertices are ordered lexicographically with respect
	to their pairs.
	
	Now we show how to obtain a $(q+k)$-queue layout using the resulting vertex order. To this end, 
	we use $q$ queues for the intra-bag edges and separate $k$ queues for inter-bag edges.
	It is easy to see that the intra-bag edges do not nest (as the bags are separated in the order); 
	therefore, we only need to verify that inter-bag edges fit in $k$ queues.
	
	Consider two edges, $e_1 = (u_1, v_1)$ and $e_2 = (u_2, v_2)$ of $G$. Since the layout is derived 
	from a $1$-queue layout of $T$, the edges may nest only when $u_1$ and $u_2$ are from the same bag;
	let $x$ be the bag such that $u_1, u_2 \in T_x$. Assign inter-bag edges rooted at $x$ to $k$ queues
	respecting the nice order of the shadows. That is, edges incident to the first vertices of the
	shadows are in the first queue, edges incident to the second vertices of the shadows are in the second
	queue and so on. Since the shadow order is nice and every shadow is of size $\le k$, there are at most $k$
	queues in the layout. 
 \end{proof}

Before applying \cref{lm:main} to 2-degenerate quadrangulations, we remark that the result provides a 
$5$-queue layout of planar 3-trees, as shown by Alam et al.~\cite{ABGKP18}. Indeed, a breadth-first search
(starting from an arbitrary vertex) on $3$-trees yields a tree-partition in which every bag is an
outerplanar graph. The shadow width of the tree-partition is $3$ (the length of each face), and it is
easy to construct a nicely ordered $2$-queue layout for every outerplanar graph~\cite{WPD04}. Thus, \cref{lm:main}
yields a $5$-queue layout for planar $3$-trees. 
Now we 
turn our attention to
2-degenerate quadrangulations.

\begin{restatable}{lemma}{bfsquad}
	\label{lm:bfs_quad}
	Every $2$-degenerate quadrangulation admits a tree-partition of shadow width $4$ such that
	every bag induces a leveled planar graph.
\end{restatable}

\begin{proof}
Recall that $2$-degenerate quadrangulations admit a recursive construction starting from a 4-cycle. At each step a vertex $v$ of degree $2$ is added inside a 4-face $f=(u_1,u_2,u_3,u_4)$ of the constructed subgraph, such that $v$ is connected to two opposite vertices of $f$, that is $v$ is connected either to $u_1$ and $u_3$ or to $u_2$ and $u_4$. This construction yields a total order $v_1<v_2<\dots<v_n$ on the vertices of the input $2$-generate quadrangulation $G$ of order $n$, such that $v_1,v_2,v_3,v_4$ are the vertices of the starting 4-cycle, and $v_{i\geq 5}$ is the vertex added at step $i$. This order is not unique, as one may permute the starting four vertices, or (possibly) select a different vertex to add at each step. Assume that vertices $u_1$, $u_2$, $\dots$ $u_k$ ($k\geq 1)$ that are added at steps $i_1<i_2<...<i_k$ (that is $u_s=v_{i_s}$), are connected to the same two vertices $v_j$ and $v_{j'}$ of $G$, with $j<j'<i_1$. Following the definitions given in Section~\ref{sec:lower-bounds}, we say that vertices $u_1$, $u_2$, $\dots$ $u_k$ are siblings with parents $v_j$ and $v_{j'}$. Note that at step $i_1$, any vertex among $u_1$, $u_2$, $\dots$ $u_k$ may be added, and in particular they can all be added at consecutive steps. 

Our goal is to assign a \emph{layer value} $\lambda(v)$  to each vertex $v$ of $G$  such that \begin{inlineenum} \item the subgraph $G_\lambda$ induced by vertices of layer $\lambda$ ($\lambda\geq0$) is a leveled planar graph, and \item the connected components of all subgraphs $G_\lambda$ define the bags of a tree-partition of shadow width $4$\end{inlineenum}. Note that we will assign layer values that do not necessarily correspond to a BFS-layering of $G$.

The four vertices of the starting $4$-cycle have layer value equal to $0$. Consider now a set of siblings $u_1,\dots,u_k$ with parents $v$ and $v'$, that are placed inside a $4$-face $f=(v,w,v',w')$ of the constructed subgraph. We compute the layer value of vertices $u_1,\dots,u_k$ as follows. Assume without loss of generality that $\lambda(v)\leq \lambda(v')$ and $\lambda(w)\leq \lambda(w')$. Further, assume that $u_1,\dots,u_k$ are such that the cyclic order of edges incident to $v$ are $(v,w), (v,u_1),(v,u_2),\dots,(v,u_k),(v,w')$ whereas the cyclic order of edges incident to $v'$ are $(v',w'), (v',u_k),(v',u_{k-1}),\dots,(v',u_1),(v',w)$; see~\cref{fig:level_faces}. We will insert these vertices in the order $u_1,u_k,u_2\dots,u_{k-1}$, that is, $u_1$ is inserted inside $f_1=f$, $u_k$ inside $f_k=(v,u_1,v',w')$ and every subsequent $u_i$ inside $f_i=(v,u_{i+1},u_k,w')$. Then the layer value of vertex $u_i$ is defined as $\lambda(u_i)=1+\min\{\lambda(x), x\in f_i\}$, for $1\leq i\leq k$. In the following, we will associate $f$ with the layer values of its vertices, that is we call $f$ a \emph{\facetype{$\lambda(v)$}{$\lambda(w)$}{$\lambda(v')$}{$\lambda(w')$}-face}, where vertices $v$ and $v'$ form a pair whose children are placed inside $f$.

Note that, by the definition of the layer values, if a vertex $u$ that is placed inside a face $f$  has layer $\lambda$, then all vertices of $f$ have layer value $\lambda-1$ or $\lambda$. This implies that connected components of layer $\lambda$ are adjacent to at most four vertices of layer value $\lambda-1$ (that form a $4$-cycle). 

To simplify the presentation, we focus on a \facetype{0}{0}{0}{0}-face $f_0$ and our goal is to determine the subgraph of layer value $1$ placed inside it; an analogous approach is used for the interior of a \facetype{$\lambda$}{$\lambda$}{$\lambda$}{$\lambda$}-face, with $\lambda>0$. 

\subparagraph{\facetype{0}{0}{0}{0}-face:} In this case, all vertices $u_1,\dots, u_k$ will have layer value $1$ (see Figure~\ref{fig:level_faces_a}). The newly created faces $(w,v,u_1,v')$ and $(w',v,u_k,v')$ are \facetype{0}{0}{1}{0}-faces, while all other faces $u_i,v,u_{i+1},v'$ for $i=1,\dots,k-1$ are $\facetype{1}{0}{1}{0}$-faces.

\subparagraph{\facetype{0}{0}{1}{0}-face:} We continue with a $\facetype{0}{0}{1}{0}$-face and then consider a $\facetype{1}{0}{1}{0}$-face. In a 
$\facetype{0}{0}{1}{0}$-face, we add only vertices $u_1$ and $u_k$, which creates the $\facetype{0}{0}{1}{1}$-faces $(w,v,u_1,v')$ and $(w,'v,u_k,v')$, and one \facetype{0}{1}{1}{1}-face $(v,u_1,v',u_k)$ (see Figure~\ref{fig:level_faces_b}). Note that the remaining siblings $u_2,\dots,u_{k-1}$ will be added as children of $v$ and $v'$ inside face $(v,u_1,v',u_k)$. 

\subparagraph{\facetype{1}{0}{1}{0}-face:} In a 
$\facetype{1}{0}{1}{0}$-face, vertices $u_1$ and $u_k$ have layer value $1$, while all other siblings $u_i$, $i=2,\dots,k-1$ have layer value $2$. Note that in a BFS-layering, vertex $u_1$ would have layer value $2$ instead of $1$. In this case, $(w,v,u_1,v')$ and $(w',v,u_k,v')$ are \facetype{0}{1}{1}{1}-faces, while the other faces cannot contain vertices of layer value $1$ (see Figure~\ref{fig:level_faces_c}). In total, we have two new types of faces, namely 
$\facetype{0}{0}{1}{1}$ and \facetype{0}{1}{1}{1}-faces. 

\subparagraph{\facetype{0}{0}{1}{1}-face:}In a $\facetype{0}{0}{1}{1}$-face, we add only vertex $u_1$, which creates two faces, namely the $\facetype{0}{0}{1}{1}$-face $(w,v,u_1,v')$, and the \facetype{0}{1}{1}{1}-face $(v,w',v',u_1)$. Note that the remaining vertices $u_2,\dots,u_{k}$ will be added as children of $v$ and $v'$ inside face $(v,w',v',u_1)$; see Figure~\ref{fig:level_faces_d}. 

\subparagraph{\facetype{0}{1}{1}{1}-face:} Finally, a \facetype{0}{1}{1}{1}-face is split into $k+1$ faces of type \facetype{1}{0}{1}{1} (see Figure~\ref{fig:level_faces_f}). In a \facetype{1}{0}{1}{1}-face, only vertex $u_1$ has layer value $1$, and face $(w,v,u_1,v')$ is of type \facetype{0}{1}{1}{1} (see Figure~\ref{fig:level_faces_e}). 

\begin{figure}
\newlength{\levelfaceswidth}
\setlength{\levelfaceswidth}{0.16\linewidth}
	\centering
	\begin{subfigure}[b]{\levelfaceswidth}
		\centering
		\includegraphics[page=1]{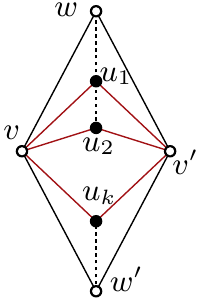}
		\caption{$\facetype{0}{0}{0}{0}$}
  \label{fig:level_faces_a}
	\end{subfigure}
	\hfill
 \begin{subfigure}[b]{\levelfaceswidth}
		\centering
		\includegraphics[page=2]{level_faces_small}
		\caption{$\facetype{0}{0}{1}{0}$}
  \label{fig:level_faces_b}
	\end{subfigure}
	\hfill
 \begin{subfigure}[b]{\levelfaceswidth}
		\centering
		\includegraphics[page=3]{level_faces_small}
		\caption{$\facetype{1}{0}{1}{0}$}
  \label{fig:level_faces_c}
	\end{subfigure}
	\hfill
 \begin{subfigure}[b]{\levelfaceswidth}
		\centering
		\includegraphics[page=4]{level_faces_small}
		\caption{$\facetype{0}{0}{1}{1}$}
  \label{fig:level_faces_d}
	\end{subfigure}
	\hfill
 \begin{subfigure}[b]{\levelfaceswidth}
		\centering
		\includegraphics[page=6]{level_faces_small}
		\caption{\facetype{0}{1}{1}{1}}
  \label{fig:level_faces_e}
	\end{subfigure}
	\hfill
 \begin{subfigure}[b]{\levelfaceswidth}
		\centering
		\includegraphics[page=5]{level_faces_small}
		\caption{\facetype{1}{0}{1}{1}}
  \label{fig:level_faces_f}
	\end{subfigure}
	\caption{Cases for faces that contain a set of children in their interior. Vertices at layer values $0$, $1$ and $2$ are drawn as black circles, black disks and red circles resp. Edges that connect vertices of the same (different) layer value are drawn black (red resp.). The dotted diagonals inside a face connect the parents of the children placed inside. Faces of type \facetype{0}{1}{1}{1} are shaded in green, and faces that contain vertices of layer value at least $2$ are shaded in gray.}
	\label{fig:level_faces}
\end{figure}

\begin{figure}
	\centering
	\begin{minipage}{0.18\textwidth}
        \centering
        \begin{subfigure}[b]{\textwidth}
            \centering
            \includegraphics[page=16]{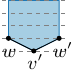}
            \caption{\nolinenumbers}
            \label{fig:level_face_0111_schematic_a_short}
            \label{fig:level_face_0111_schematic_a}
        \end{subfigure}
        
        \vspace{2ex}
        \begin{subfigure}[b]{\textwidth}
            \centering
            \includegraphics[page=17]{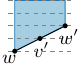}
            \caption{\nolinenumbers}
            \label{fig:level_face_0111_schematic_b_short}
            \label{fig:level_face_0111_schematic_b}
        \end{subfigure}
        
        \vspace{2ex}
        \begin{subfigure}[b]{\textwidth}
            \centering
            \includegraphics[page=18]{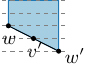}
            \caption{\nolinenumbers}
            \label{fig:level_face_0111_schematic_c_short}
            \label{fig:level_face_0111_schematic_c}
        \end{subfigure}
	\end{minipage}
	\begin{minipage}{0.8\textwidth}
	\begin{minipage}{.35\textwidth}
            \centering
            \begin{subfigure}[b]{\textwidth}
                \centering
                \includegraphics[page=21]{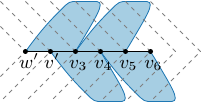}
                \caption{\nolinenumbers}
                \label{fig:level_face_0011_schematic_short}
            \end{subfigure}
            
            \begin{subfigure}[b]{\textwidth}
                \centering
                \includegraphics[page=22]{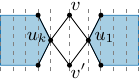}
                \caption{\nolinenumbers}
                \label{fig:level_face_1010_schematic_a_short}
            \end{subfigure}
            
            \begin{subfigure}[b]{\textwidth}
                \centering
                \includegraphics[page=23]{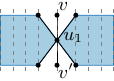}
                \caption{\nolinenumbers}
                \label{fig:level_face_1010_schematic_b_short}
            \end{subfigure}
        \end{minipage}
        \begin{minipage}{.64\textwidth}
            \centering
            \begin{subfigure}[b]{0.45\textwidth}
                \centering
                \includegraphics[page=24]{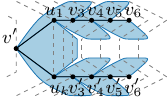}
                \caption{\nolinenumbers}
                \label{fig:level_face_0010_schematic_a_short}
            \end{subfigure}\hfill
            \begin{subfigure}[b]{0.45\textwidth}
                \centering
                \includegraphics[page=25]{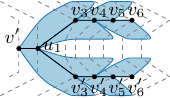}
                \caption{\nolinenumbers}
                \label{fig:level_face_0010_schematic_b_short}
            \end{subfigure}

            \begin{subfigure}[b]{\textwidth}
                \centering
                \includegraphics[page=26]{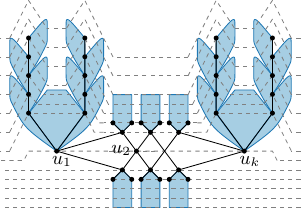}
                \caption{\nolinenumbers}
                \label{fig:level_face_0000_schematic_short}
            \end{subfigure}
        \end{minipage}
	\end{minipage}
	\caption{Illustrations of $\Gamma_f$  when $f$ is  (a)--(c)~a \facetype{0}{1}{1}{1}-face, (d)~a \facetype{0}{0}{1}{1}-face, (e)--(f)~a \facetype{1}{0}{1}{0}-face, (g)--(h)~a \facetype{0}{0}{1}{0}-face, and (i)~a \facetype{0}{0}{0}{0}-face.}
	\label{fig:level_face_0111_schematic_short}
	\label{fig:level_face_0111_schematic}
\end{figure}

In order to create a layered planar drawing, we first consider faces of type \facetype{0}{1}{1}{1}, since they are contained in almost all other types of faces. The existence of $k$ children in such a face $f$ creates $k+1$ faces of type \facetype{1}{0}{1}{1}. Let $g_i$ be the face $(u_i,v,u_{i+1},v')$ of type \facetype{1}{0}{1}{1}, for $i=0,\dots,k$, where $u_0=w$ and $u_{k+1}=w'$. If another set of children $u_1',\dots,u'_{k'}$ is added inside $g_i$ then their parents are vertices $u_i$ and $u_{i+1}$ and only $u_1'$ has layer value $1$.  Hence the addition of $u'_1$ creates a new \facetype{0}{1}{1}{1}-face $g_i'$ (namely face $(v,u_i,u_1',u_{i+1})$). Further addition of children inside $g_i'$, will split $g_i'$ into faces of type \facetype{1}{0}{1}{1}. So, let $G_f$ be the subgraph induced by all vertices of layer value $1$ inside a \facetype{0}{1}{1}{1}-face $f$ (including its boundary vertices). 

\begin{restatable}{claim}{zerooneoneone}\label{claim:level_face_0111_draw}
The subgraph $G_f$ of a \facetype{0}{1}{1}{1}-face $f=(v,w,v',w')$ has a leveled planar drawing such that level $0$ contains one vertex among $w$, $v'$ and $w'$ (and no other vertex); see Figures~\ref{fig:level_face_0111_schematic_a_short}-\ref{fig:level_face_0111_schematic_c_short}.
\end{restatable}

\begin{claimproof}
Based on the previous observation, we will create a sequence $\{G_f^i\}_{1\leq i\leq s}$ such that $G_f^i$ is a subgraph of $G_f^{i+1}$ and $G_f^s$ is $G_f$. Let $H_f^i$ be the subgraph of $G$ induced by the vertices of $G_f^i$ and vertex $v$. In our construction, graphs $G_f^i$ will have the following properties: \begin{inlineenum}
\item\label{prop:0111:1} the interior faces of $G_f^i$ are of type \facetype{1}{1}{1}{1}, \item\label{prop:0111:2} $H_f^i$ contains all \facetype{1}{1}{1}{1}-faces of $G_f^i$ and every other interior face is of type \facetype{1}{0}{1}{1} with vertex $v$ on its boundary.
\end{inlineenum} 

At the first step $G_f^1$ consists of vertices $v'$, $w$, $w'$ and the $k_{v'}$ siblings of $v$ and $v'$. The subgraph so far is a star with $v'$ as center. It is not hard to see the $G_f^1$ satisfies Properties~\eqref{prop:0111:1} and~\eqref{prop:0111:2}. At step $i$ we select a \facetype{1}{0}{1}{1}-face $f_i$ of $H_f^{i-1}$ (that is not empty). Let $f_i=(v,u_i,v_i,u'_i)$, and let $w_i$ be the only child of layer value $1$ inside $f_i$ (with parents $v$ and $v_i$). Then $f_i$ is split into the \facetype{1}{1}{1}{1}-face $(v,w_i,v_i,u'_i)$ and the \facetype{0}{1}{1}{1}-face $f'_i=(v,u_i,w_i,u'_i)$. Further let $x_1,\dots x_k$, $k\geq 0$, be the children inside $f'_i$. Recall that $x_1,\dots x_k$ have layer value $1$ and their parents are $v$ and $w_i$. We obtain $G_f^i$ from $G_f^{i-1}$ by adding a star with center $w_i$ and vertices $x_1,\dots x_k$ as leafs, and by connecting $w_i$ to vertices $u_i$ and $u'_i$ so that vertices $x_1,\dots x_k$ are on the outer face of $G_f^i$. As face $f_i$ of $H_f^{i-1}$ is split into a \facetype{1}{1}{1}{1}-face and $k+1$ \facetype{1}{0}{1}{1}-faces, it follows that $G_f^i$ contains one more interior \facetype{1}{1}{1}{1}-face than $G_f^{i-1}$ (that is Property~\eqref{prop:0111:1} is satisfied), and all other interior faces of $H_f^{i}$ are of type \facetype{1}{0}{1}{1} with $v$ on their boundary (Property~\eqref{prop:0111:2}).

Now we show how to construct the leveled planar drawing of $G_f$ based on the sequence $\{G_f^i\}_{1\leq i\leq s}$. In particular, we will extend a leveled planar drawing $\Gamma_f^{i-1}$ of $G_f^{i-1}$ to a leveled planar drawing $\Gamma_f^{i}$ of $G_f^i$, for $1< i\leq s$. For every \facetype{1}{0}{1}{1}-face $f_j$ of $H_f^{i-1}$ (for some $j\geq i$) with vertices $v,u_j,v_j,u'_j$ we denote as $P_j$ the path $(u_j,v_j,u'_j)$, and say that $P_j$ is the \emph{boundary path} of $f_j$. Note that $P_j$ is along the outer face of $\Gamma_f^{i-1}$. We say that $P_j$ forms a \emph{small angle} in $\Gamma_f^{i-1}$ if $\ell(u_j)=\ell(u'_j)=\ell(v_j)+1$ holds, where $\ell(u)$ is the level of vertex $u$; see Figure~\ref{fig:level_face_small_angle}. We also say that $P_j$ forms a \emph{large angle} if either $\ell(v_j)=\ell(u_j)+1=\ell(u'_j)-1$ or $\ell(v_j)=\ell(u'_j)+1=\ell(u_j)-1$ holds; see Figures~\ref{fig:level_face_large_angle_a} and \ref{fig:level_face_large_angle_b}. Our algorithm maintains the following invariants: \begin{inlineenum}
\item\label{inv:0111:1} level $0$ contains one of vertices $w$, $v'$ or $w'$ of $f$ (and no other vertices), \item\label{inv:0111:2} vertices $v_i$ and $x_1,\dots,x_k$ of $G_f^i$ are placed in the outer face of  $\Gamma_f^{i-1}$, and \item\label{inv:0111:3} every path $P_j$ of $G_f^i$ forms either a small or a large angle along the outer face of  $\Gamma_f^i$
\end{inlineenum}. 

\begin{figure}
	\centering
	\begin{subfigure}[b]{0.3\linewidth}
		\centering
		\includegraphics[page=7]{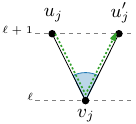}
		\caption{\nolinenumbers}
 \label{fig:level_face_small_angle}
	\end{subfigure}
\begin{subfigure}[b]{0.3\linewidth}
		\centering
		\includegraphics[page=8]{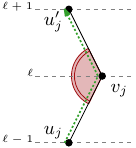}
		\caption{\nolinenumbers}
 \label{fig:level_face_large_angle_a}
	\end{subfigure}
\begin{subfigure}[b]{0.3\linewidth}
		\centering
		\includegraphics[page=9]{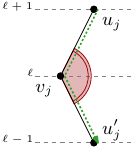}
		\caption{\nolinenumbers}
 \label{fig:level_face_large_angle_b}
	\end{subfigure}
	\caption{The boundary path $P_j=u_j,w_j,u'_j$ forms a (a) small angle, (b-c) a large angle in $\Gamma _f^i$. The boundary path is drawn green; horizontal dashed lines indicate levels.}
	\label{fig:level_face_angles}
\end{figure}

Three different leveled planar drawings of $G_f^1$ are shown in Figure~\ref{fig:level_face_0111_base}. We have that $f_1=(v,u_1,w_1,u'_1)$, where $u_1=w$, $u_1'=w'$ and $w_1=v'$ (refer to Figure~\ref{fig:level_faces_f}). It is not hard to see that Invariants~\eqref{inv:0111:1}, \eqref{inv:0111:2} and~\eqref{inv:0111:3} are satisfied in all three drawings. 
\begin{figure}
	\centering
	\begin{subfigure}[b]{0.3\linewidth}
		\centering
		\includegraphics[page=10]{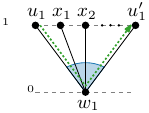}
		\caption{\nolinenumbers}
 \label{fig:level_face_0111_base_a}
	\end{subfigure}
\begin{subfigure}[b]{0.3\linewidth}
		\centering
		\includegraphics[page=11]{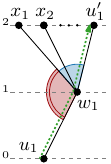}
		\caption{\nolinenumbers}
 \label{fig:level_face_0111_base_b}
	\end{subfigure}
\begin{subfigure}[b]{0.3\linewidth}
		\centering
		\includegraphics[page=12]{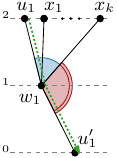}
		\caption{\nolinenumbers}
 \label{fig:level_face_0111_base_c}
	\end{subfigure}
	\caption{Possible drawings for $G_f^1$.}
	\label{fig:level_face_0111_base}
\end{figure}

So, assume that we have constructed the drawing $\Gamma_f^{i-1}$ for $G_f^{i-1}$ satisfying the invariants. $G_f^i$ is obtained from $G_f^{i-1}$ by adding a star with center $w_i$, leafs $x_1,\dots,x_k$, and such that $w_i$ is connected to the endpoints of the boundary path $P_i$ of $f_i$. By construction, the new vertices are added in the exterior of $\Gamma_f^{i-1}$, satisfying Invariant~\eqref{inv:0111:2}.
\begin{figure}
	\centering
	\begin{subfigure}[b]{0.3\linewidth}
		\centering
		\includegraphics[page=13]{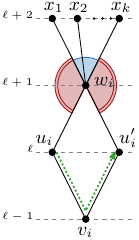}
		\caption{\nolinenumbers}
 \label{fig:level_face_0111_step_small}
	\end{subfigure}
\begin{subfigure}[b]{0.3\linewidth}
		\centering
		\includegraphics[page=14]{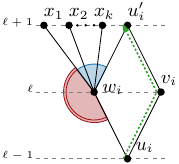}
		\caption{\nolinenumbers}
 \label{fig:level_face_0111_step_large_a}
	\end{subfigure}
\begin{subfigure}[b]{0.3\linewidth}
		\centering
		\includegraphics[page=15]{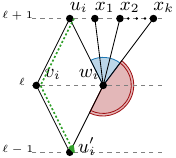}
		\caption{\nolinenumbers}
 \label{fig:level_face_0111_step_large_b}
	\end{subfigure}
	\caption{Extending $\Gamma_f^{i-1}$ to $\Gamma_f^i$, when $P_i$ forms a (a)~small angle, or (b)--(c)~ a large angle.}
	\label{fig:level_face_0111_step}
\end{figure}

We consider two cases depending on whether $P_i$ forms a small or a large angle in $\Gamma_f^{i-1}$. In the first case, let $\ell=\ell(u_i)=\ell(u'_i)$. We place $w_i$ at level $\ell+1$ and vertices $x_1,\dots,x_k$ at level $\ell+2$. 
In the constructed drawing the newly added vertices are placed at levels different from level $0$, satisfying Invariant~\eqref{inv:0111:1}, while the new boundary paths $(u_i,w_i,x_1)$ and $(x_k,w_i,u'_i)$ form large angles and all other form small angles, satisfying Invariant~\eqref{inv:0111:3}; see Figure~\ref{fig:level_face_0111_step_small}. 
Note that in the case where $k=0$, that is there are no child vertices inside $f_i$, the only vertex of $G_f$ inside $f_i$ is vertex $w_i$, and no new boundary paths are created. In the second case, let $\ell$ be the level of vertex $v_i$. We place $w_i$ at level $\ell$ and vertices $x_1,\dots, x_k$ at level $\ell+1$; refer to Figures~\ref{fig:level_face_0111_step_large_a} and \ref{fig:level_face_0111_step_large_b}. 
The newly added vertices are placed at layers different from level $0$, while  $(u_i,w_i,x_1)$ or  $(x_k,w_i,u'_i)$ forms a large angle, and all other new boundary paths  form small angles. Therefore Invariants~\eqref{inv:0111:1} and~\eqref{inv:0111:3} are satisfied in this case as well. Therefore, the constructed drawing of $G_f$ satisfies Invariant~\eqref{inv:0111:1} and the claim follows. 
\end{claimproof}

Figure~\ref{fig:level_face_0111_schematic} illustrates a schematic representation of different leveled planar drawings of $G_f$ induced by the layer value-$1$ vertices of a \facetype{0}{1}{1}{1}-face $f$, depending of the initial placement of vertices $v'$, $w$ and $w'$.
Now we turn our attention to faces of type $\facetype{0}{0}{1}{1}$. We will prove that such a face can be split into an empty $\facetype{0}{0}{1}{1}$-face and a series of \facetype{0}{1}{1}{1}-faces.


\begin{claim}\label{claim:level_face_0011_decompose}
    Let $f=v,w,v',w'$ be a $\facetype{0}{0}{1}{1}$-face. There exists a path $(v_1,v_2,\dots,v_{s+2})$ of $s+2$ layer value-$1$ vertices ($s\geq0$), such that the following hold:
    \begin{itemize}
        \item $v_1=w'$ and $v_2=v'$.
        \item Vertex $v_i$ is adjacent to $v$ ($w$) for odd (even, resp.) $i$, $1\leq i\leq s+2$.
        \item Face $f$ is split into $s$ faces $\{f_i\}_{1\leq i\leq s}$ of type \facetype{0}{1}{1}{1} and one empty $\facetype{0}{0}{1}{1}$-face $g$, where:
        \begin{itemize}
            \item $f_{2j}=(w,v_{2j+2},v_{2j+1},v_{2j})$, for $1\leq j\leq s/2$,  
            \item  $f_{2j-1}=(v,v_{2j-1},v_{2j},{v_{2j+1}})$, for $1\leq j\leq (s+1)/2$, and 
            \item $g=(v,w,v_{s+2},v_{s+1})$, if $s$ is even, or $g=(w,v,v_{s+2},v_{s+1})$ if $s$ is odd.
        \end{itemize}
    \end{itemize}
\end{claim}
\begin{claimproof}
Let $v_1=w'$ and $v_2=v'$. We will prove the claim using induction on the number $k$ of layer value-$1$ vertices inside $f$. If $f$ contains no layer value-$1$ vertices (that is $k=0$ and $f$ is empty), then $s=0$ and the claim holds with $g=f$. Assume that the claim holds for $k'<k$ layer value-$1$ vertices and that $f$ contains $k$ layer value-$1$ vertices. Then, vertex $u_1$ of Figure~\ref{fig:level_faces_d} (which  has layer value $1$) exists and is a child inside $f$ with parents $v$ and $v'$. We let $v_3=u_1$. Now $f$ is split into face $f_1=(v,v_1,v_2,v_3)$ which is of type \facetype{0}{1}{1}{1},  and the $\facetype{0}{0}{1}{1}$-face $g_1=(w,v,v_3,v_2)$. As $g_1$ contains at most $k-1$ layer value-$1$  vertices, $g_1$ contains a path $(v'_1,\dots, v'_{s'+2})$ of layer value-$1$ vertices that split $g_1$ into $s'$ \facetype{0}{1}{1}{1}-faces $f'_i$ ($i=1,\dots,s'$) and an empty $\facetype{0}{0}{1}{1}$-face $g'$ that satisfy the claim. For convenience, let $g_1=(v'',w'',v'_2,v'_1)$, where $v''=w$, $w''=v$, $v'_2=v_3$ and $v'_1=v_2$. We set $v_{i+1}=v'_i$, for $i=1,\dots,s'+2$ and we will prove that the path $v_1,\dots v_{s'+3}$ satisfies the properties of the claim. By definition, we have that $v_1=w'$ and $v_2=v'$. Also, in $g_1$, vertex $v'_i$ is connected to $v''$ if $i$ is odd and to $w''$ if $i$ is even. Hence $v_j=v'_{j-1}$ is connected to $w=v''$ when $j$ is even and to $v=w''$ when $j$ is odd, as required. For the faces we have that $f_1=(v,v_1,v_2,v_3)$, and we set $f_{i+1}=f'_i$, $i=1,\dots,s'$, and $g=g'$. 
We have that $f_{2j}=f'_{2j-1}=(v'',v'_{2j-1},v'_{2j},{v'_{2j+1}})=(w,v_{2j},v_{2j+1},v_{2j+2})$ and $f_{2j-1}=f'_{2j-2}=(w'',v'_{2j},v'_{2j-1},v'_{2j-2})=(v,v_{2j+1},v_{2j},{v_{2j-1}})$. On the other hand, as $s=s'+1$, $g=g'=(w'',v'',v'_{s'+2},v'_{s'+1})=(v,w,v_{s+2},v_{s+1})$, if $s'$ is odd and $s$  even, and  $g=g'=(v'',w'',v'_{s'+2},v'_{s'+1})=(w,v,v_{s+2},v_{s+1})$ if $s'$ is even and $s$ is odd. Hence the conditions of the claim hold. An example for $s=3$ and $s=4$ is shown in  Figures~\ref{fig:level_face_0011_decompose_odd} and ~\ref{fig:level_face_0011_decompose_even}.  
\end{claimproof}	

\begin{figure}
	\centering
	\begin{subfigure}[b]{0.3\linewidth}
		\centering
		\includegraphics[page=19]{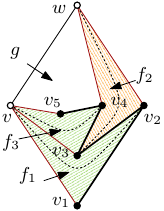}
		\caption{$s=3$}
 \label{fig:level_face_0011_decompose_odd}
	\end{subfigure}
\begin{subfigure}[b]{0.3\linewidth}
		\centering
		\includegraphics[page=20]{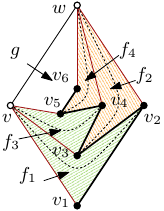}
		\caption{$s=4$}
 \label{fig:level_face_0011_decompose_even}
	\end{subfigure}
\begin{subfigure}[b]{0.33\linewidth}
		\centering
		\includegraphics[page=21]{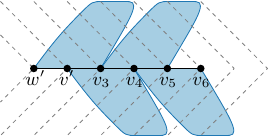}
		\caption{\nolinenumbers}
 \label{fig:level_face_0011_schematic}
	\end{subfigure}
	\caption{(a-b)~Splitting a $\facetype{0}{0}{1}{1}$-face $f$into a series of \facetype{0}{1}{1}{1}-faces (highlighted in green and orange) and an empty $\facetype{0}{0}{1}{1}$-face, (c)~Schematic representations of $\Gamma_f$.}
	\label{fig:level_face_0011_decompose}
\end{figure}

In the following we compute a leveled planar drawing of a $\facetype{0}{0}{1}{1}$-face.

\begin{claim}\label{claim:level_face_0011_draw}
The subgraph $G_f$ of a $\facetype{0}{0}{1}{1}$-face $f=v,w,v',w'$ has a leveled planar drawing such that level $0$ contains only vertex $w'$; see Figure~\ref{fig:level_face_0011_schematic}.
\end{claim}
\begin{claimproof}
For every face $f_i$ ($i\geq 1$) of Claim~\ref{claim:level_face_0011_decompose}, we create a leveled planar drawing $\Gamma_i$ using Claim~\ref{claim:level_face_0111_draw}, such that $\Gamma_i$ has only vertex $v_i$ on level $i-1$ using the drawings of Figures~\ref{fig:level_face_0111_schematic_b} and \ref{fig:level_face_0111_schematic_c} for odd and even $i$, respectively. On each level $\ell$ the vertices are ordered as follows. Vertices of the same face $f_i$ (that are on level $\ell$) appear consecutively. For $i\neq j\geq 1$,  let $u_i$ be a vertex inside face $f_i$ and $u_j$ a vertex inside $f_j$. If $i$ is odd and $j$ is even, then $u_i$ appears before $u_j$ along $\ell$ (from left to right); if both $i$ and $j$ are odd (even) with $i<j$, then $u_i$ precedes (follows, resp.) $u_j$. Note that the derived drawing has only vertex $w'$ on level $0$ as claimed.
 \end{claimproof}

Next, we focus on faces of type $\facetype{1}{0}{1}{0}$. As the only layer value-$1$ vertices inside such faces belong to \facetype{0}{1}{1}{1}-faces, we have the following.

\begin{claim}
\label{claim:level_face_1010_draw}
The subgraph $G_f$ of a $\facetype{1}{0}{1}{0}$-face $f=v,w,v',w'$ has a leveled planar drawing such that level $0$ contains only vertices $v$ and $v'$, vertices of the \facetype{0}{1}{1}{1}-face $w,v,u_1,v'$ are drawn on levels above level $0$, while vertices of the \facetype{0}{1}{1}{1}-face $w',v,u_k,v'$ are drawn on levels below level $0$. In the special case where $u_1=u_k$, then also $u_1$ is on level $0$; see Figure~\ref{fig:level_face_1010_schematic}.
\end{claim}
\begin{claimproof}
    We draw the \facetype{0}{1}{1}{1}-face $(w,v,u_1,v')$ using Claim~\ref{claim:level_face_0111_draw} and such that vertex $u_1$ is the only vertex at level $1$ (or $0$ if $u_1=u_k$) and all other vertices are at levels greater than $1$ (or $0$, resp.), as in Figure~\ref{fig:level_face_0111_schematic_a}. Similarly for face $(w',v,u_k,v')$ we create a leveled planar drawing with  $u_k$ at level $-1$ (or $0$ if $u_1=u_k$) and all other vertices at levels below $-1$ (or $0$, resp.). Vertices $v$ and $v'$ are placed at level $0$ as shown in Figure~\ref{fig:level_face_1010_schematic_a}, if $u_1\neq u_k$, otherwise they are placed together with $u_1=u_k$ and such that $u_1$ is  between $v$ and $v'$; see Figure~\ref{fig:level_face_1010_schematic_b}.
\end{claimproof}

\begin{figure}[!ht]
	\centering
	\begin{subfigure}[b]{0.3\linewidth}
		\centering
		\includegraphics[page=22]{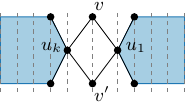}
		\caption{\nolinenumbers}
 \label{fig:level_face_1010_schematic_a}
	\end{subfigure}
\begin{subfigure}[b]{0.3\linewidth}
		\centering
		\includegraphics[page=23]{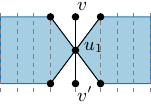}
		\caption{\nolinenumbers}
 \label{fig:level_face_1010_schematic_b}
	\end{subfigure}
	\caption{Schematic representations of $\Gamma_f$ when $f$ is a $\facetype{1}{0}{1}{0}$-face with (a)~$u_1\neq u_k$ and (b)~$u_1=u_k$.}
	\label{fig:level_face_1010_schematic}
\end{figure}

For a face $f$ of type \facetype{0}{0}{1}{0}, recall that it consists of three faces; face $(w,v,u_1,v')$, which is of type $\facetype{0}{0}{1}{1}$, face $(v,u_1,v',u_k)$ of type \facetype{0}{1}{1}{1} and face $(w',v,u_k,v')$ of type $\facetype{0}{0}{1}{1}$. 

\begin{claim}
\label{claim:level_face_0010_draw}
The subgraph $G_f$ of a \facetype{0}{0}{1}{0}-face $f=v,w,v',w'$ has a leveled planar drawing such that level $0$ contains only vertex $v'$; see Figure~\ref{fig:level_face_0010_schematic}.
\end{claim}
\begin{claimproof}
    We use Claim~\ref{claim:level_face_0011_draw} to produce a leveled planar drawing of face $(w,v,u_1,v')$ with vertex $v'$ on level $0$ as shown in Figure~\ref{fig:level_face_0011_schematic}. A similar drawing is obtained for $(w',v,u_k,v')$ with $v'$ on level $0$. Now face $(v,u_1,v',u_k)$ is drawn as in Figure~\ref{fig:level_face_0111_schematic_a} with $v'$ on level $0$, using Claim~\ref{claim:level_face_0111_draw}. The three drawings can be glued together as depicted in Figure~\ref{fig:level_face_0010_schematic_a}, or Figure~\ref{fig:level_face_0010_schematic_b} depending on whether $u_1\neq u_k$ or $u_1=u_k$ (in which case the $\facetype{0}{1}{1}{1}$-face $(v,u_1,v',u_k)$  does not exist).
\end{claimproof}

\begin{figure}
	\centering
	\begin{subfigure}[b]{0.45\linewidth}
		\centering
		\includegraphics[page=24]{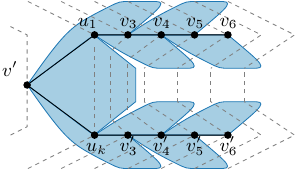}
		\caption{\nolinenumbers}
 \label{fig:level_face_0010_schematic_a}
	\end{subfigure}
\begin{subfigure}[b]{0.45\linewidth}
		\centering
		\includegraphics[page=25]{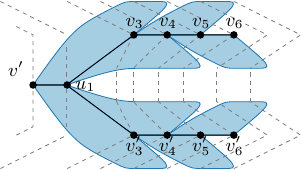}
		\caption{\nolinenumbers}
 \label{fig:level_face_0010_schematic_b}
	\end{subfigure}
	\caption{Schematic representations of $\Gamma_f$ when $f$ is a \facetype{0}{0}{1}{0}-face with (a)~$u_1\neq u_k$ and (b)~$u_1=u_k$.}
	\label{fig:level_face_0010_schematic}
\end{figure}

The only type of faces that we have not considered yet are $\facetype{0}{0}{0}{0}$-faces. We combine leveled planar drawings for faces $(w,v,u_1,v')$, $(w',v,u_k,v')$ and  faces $f_i=(u_i,v,u_{i+1},v')$ for $i=1,\dots,k-1$. We create drawings for the \facetype{0}{0}{1}{0}-faces $(w,v,u_1,v')$ and $(w',v,u_k,v')$ using Claim~\ref{claim:level_face_0010_draw}, such that $u_1$ and $u_k$ are the only vertices placed at level $0$. Then for every face $f_i$ ($i=1,\dots,k-1$), we use Claim~\ref{claim:level_face_1010_draw}; vertices $u_i$, for $i=1,\dots,k$ are placed at level $0$. We combine the drawings as shown in Figure~\ref{fig:level_face_0000_schematic}, by ordering the vertices on each level as follows. Vertices at level $\ell$ that belong to the same face appear consecutively along $\ell$. Vertices of face $(w,v,u_1,v')$, precede all vertices of faces $f_i$ for $i=1,\dots,k-1$, and vertices of $(w',v,u_k,v')$ appear last along $\ell$. For two faces $f_i$ and $f_j$ with $1\leq i<j\leq k-1$, we have that vertices of $f_i$ precede vertices of $f_j$. 

\begin{figure}
\centering
	\includegraphics[page=26]{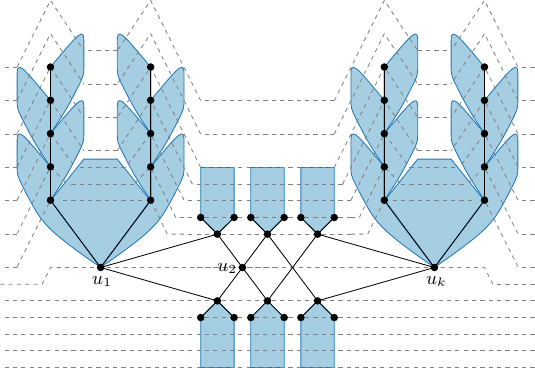}
	\caption{Schematic representations of $\Gamma_f$ when $f$ is a $\facetype{0}{0}{0}{0}$-face.}
	\label{fig:level_face_0000_schematic}
\end{figure}

So far, we focused on a $\facetype{0}{0}{0}{0}$ face $f$ and determined a leveled planar drawing of $G_f$, which is the subgraph induced by vertices of level $1$ inside $f$. Clearly, starting from any face with all vertices having the same layer value $\lambda$, we can compute a leveled planar drawing of the layer value-($\lambda+1$) subgraph of $G$ that is inside this face. Now we are ready to compute the tree-partition of a 2-generate quadrangulation $G$. 
We assign layer value equal to $0$ to  the vertices on the outer face of $G$, and compute the layer value of all other vertices. Let $G_\lambda$ denote the subgraph of $G$ induced by vertices of layer value $\lambda$ ($\lambda\geq 0$). We have that all edges of $G$ are either level edges (that is, belong to $G_\lambda$ for some value of $\lambda$), or connect subgraphs of consecutive layer values $\lambda$ and $\lambda+1$. In particular, each connected component $H_{\lambda+1}$ of $G_{\lambda+1}$ is located inside a $4$-face $f(H_{\lambda+1})$  of $G_\lambda$, and the vertices of $f(H_{\lambda+1})$ are the only vertices of $G_\lambda$ that are connected to the vertices  of that connected component of $H_{\lambda+1}$. Now, for each value of $\lambda$, we put the connected components of $G_\lambda$ into separate bags, and therefore each bag contains a leveled planar graph. For a bag that contains connected component $H_{\lambda}$, we define its parent to be the bag that contains the component where $f(H_\lambda)$ belongs to. As each connected component $H_{\lambda}$ lies in the interior of a single face $f(H_\lambda)$, the defined bags create a tree $T$ with root-bag consisting of the outer vertices of $G$ (with layer value $0$). Additionally, the shadow of each bag consists of at most four vertices, and therefore the shadow width of $T$ is at most $4$. The lemma follows.
\end{proof}

\degenerate*

\begin{proof}
	Combine \cref{lm:bfs_quad} with \cref{lm:main}. Observe that \emph{every} leveled planar graph admits a $1$-queue layout 
	in which the faces are nicely ordered with respect to the layout. That implies 
	that we can apply \cref{lm:main} with $q=1$ (the queue number of the bags) and $k=4$ (the shadow width
	of the tree-partition).
 \end{proof}	

\section{Open Questions}

In this work, we focused on the queue number of bipartite planar graphs and related subfamilies. 
Next we highlight a few questions for future work. 
\begin{itemize}
\item First, there is still a significant gap between our lower and upper bounds for the queue number of bipartite planar graphs. 

\item Second, although $2$-degenerate quadrangulations always admit $5$-queue layouts, the question of determining their exact queue number remains open. 

\item Third, for stacked quadrangulations, our upper bound relies on the strong product theorem. We believe that a similar approach as for $2$-degenerate quadrangulations could lead to a significant improvement. 

\item Perhaps the most intriguing questions are related to mixed linear layouts of bipartite planar graphs: One may ask what is the minimum $q$ so that each bipartite planar graph admits a $1$-stack $q$-queue layout. 

\item Finally, the recognition of $1$-stack $1$-queue graphs remains an important open problem even for bipartite planar graphs.
\end{itemize}

\bibliographystyle{splncs04}
\bibliography{references}

\begin{thebibliography}{10}
\providecommand{\url}[1]{\texttt{#1}}
\providecommand{\urlprefix}{URL }
\providecommand{\doi}[1]{https://doi.org/#1}

\bibitem{DBLP:conf/gd/AlamBG0P20}
Alam, J.M., Bekos, M.A., Gronemann, M., Kaufmann, M., Pupyrev, S.: Lazy queue
  layouts of posets. In: Auber, D., Valtr, P. (eds.) Graph Drawing and Network
  Visualization 2020. Lecture Notes in Computer Science, vol. 12590, pp.
  55--68. Springer (2020). \doi{10.1007/978-3-030-68766-3\_5}

\bibitem{ABGKP18}
Alam, J.M., Bekos, M.A., Gronemann, M., Kaufmann, M., Pupyrev, S.: Queue
  layouts of planar 3-trees. Algorithmica pp. 1--22 (2020).
  \doi{10.1007/s00453-020-00697-4}

\bibitem{ABKM20}
Angelini, P., Bekos, M.A., Kindermann, P., Mchedlidze, T.: On mixed linear
  layouts of series-parallel graphs. Theor. Comput. Sci.  \textbf{936},
  129--138 (2022). \doi{10.1016/j.tcs.2022.09.019}

\bibitem{AG11}
Auer, C., Glei{\ss}ner, A.: Characterizations of deque and queue graphs. In:
  Kolman, P., Kratochv{\'i}l, J. (eds.) Graph-Theoretic Concepts in Computer
  Science. pp. 35--46. Springer Berlin Heidelberg, Berlin, Heidelberg (2011)

\bibitem{BDDEW15}
Bannister, M.J., Devanny, W.E., Dujmovi{\'c}, V., Eppstein, D., Wood, D.R.:
  Track layouts, layered path decompositions, and leveled planarity.
  Algorithmica  \textbf{81}(4),  1561--1583 (2019).
  \doi{10.1007/s00453-018-0487-5}

\bibitem{DBLP:journals/siamcomp/BattistaFP13}
Battista, G.D., Frati, F., Pach, J.: On the queue number of planar graphs.
  {SIAM} J. Comput.  \textbf{42}(6),  2243--2285 (2013).
  \doi{10.1137/130908051}

\bibitem{BGR22}
Bekos, M., Gronemann, M., Raftopoulou, C.N.: An improved upper bound on the
  queue number of planar graphs. Algorithmica pp. 1--19 (2022).
  \doi{10.1007/s00453-022-01037-4}

\bibitem{DBLP:journals/algorithmica/BekosBKR17}
Bekos, M.A., Bruckdorfer, T., Kaufmann, M., Raftopoulou, C.N.: The book
  thickness of 1-planar graphs is constant. Algorithmica  \textbf{79}(2),
  444--465 (2017). \doi{10.1007/s00453-016-0203-2}

\bibitem{DBLP:journals/corr/abs-2204-11495}
Bekos, M.A., {Da Lozzo}, G., Hlinen{\'{y}}, P., Kaufmann, M.: Graph product
  structure for h-framed graphs. CoRR  \textbf{abs/2204.11495} (2022).
  \doi{10.48550/arXiv.2204.11495}

\bibitem{DBLP:journals/siamcomp/BekosFGMMRU19}
Bekos, M.A., F{\"{o}}rster, H., Gronemann, M., Mchedlidze, T., Montecchiani,
  F., Raftopoulou, C.N., Ueckerdt, T.: Planar graphs of bounded degree have
  bounded queue number. {SIAM} J. Comput.  \textbf{48}(5),  1487--1502 (2019).
  \doi{10.1137/19M125340X}

\bibitem{DBLP:journals/jocg/KaufmannBKPRU20}
Bekos, M.A., Kaufmann, M., Klute, F., Pupyrev, S., Raftopoulou, C.N., Ueckerdt,
  T.: Four pages are indeed necessary for planar graphs. J. Comput. Geom.
  \textbf{11}(1),  332--353 (2020). \doi{10.20382/jocg.v11i1a12}

\bibitem{DBLP:journals/jct/BernhartK79}
Bernhart, F., Kainen, P.C.: The book thickness of a graph. J. Comb. Theory,
  Ser. {B}  \textbf{27}(3),  320--331 (1979).
  \doi{10.1016/0095-8956(79)90021-2}

\bibitem{DBLP:journals/jgaa/BhoreGMN22}
Bhore, S., Ganian, R., Montecchiani, F., N{\"{o}}llenburg, M.: Parameterized
  algorithms for queue layouts. J. Graph Algorithms Appl.  \textbf{26}(3),
  335--352 (2022). \doi{10.7155/jgaa.00597}

\bibitem{DBLP:journals/jgaa/BiedlSWW99}
Biedl, T.C., Shermer, T.C., Whitesides, S., Wismath, S.K.: Bounds for
  orthogonal {3D} graph drawing. J. Graph Algorithms Appl.  \textbf{3}(4),
  63--79 (1999). \doi{10.7155/jgaa.00018}

\bibitem{DBLP:conf/swat/BoseMO22}
Bose, P., Morin, P., Odak, S.: An optimal algorithm for product structure in
  planar graphs. In: Czumaj, A., Xin, Q. (eds.) {SWAT} 2022. LIPIcs, vol.~227,
  pp. 19:1--19:14. Schloss Dagstuhl - Leibniz-Zentrum f{\"{u}}r Informatik
  (2022). \doi{10.4230/LIPIcs.SWAT.2022.19}

\bibitem{DBLP:journals/corr/abs-2206-02395}
Campbell, R., Clinch, K., Distel, M., Gollin, J.P., Hendrey, K., Hickingbotham,
  R., Huynh, T., Illingworth, F., Tamitegama, Y., Tan, J., Wood, D.R.: Product
  structure of graph classes with bounded treewidth. CoRR
  \textbf{abs/2206.02395} (2022). \doi{10.48550/arXiv.2206.02395}

\bibitem{CLR87}
Chung, F.R.K., Leighton, F.T., Rosenberg, A.L.: Embedding graphs in books: A
  layout problem with applications to {VLSI} design. SIAM Journal on Algebraic
  and Discrete Methods  \textbf{8}(1),  33--58 (1987)

\bibitem{CKN19}
de~Col, P., Klute, F., N{\"o}llenburg, M.: Mixed linear layouts: Complexity,
  heuristics, and experiments. In: Archambault, D., T{\'o}th, C.D. (eds.) Graph
  Drawing and Network Visualization. pp. 460--467. Springer International
  Publishing, Cham (2019). \doi{10.1007/978-3-030-35802-0\_35}

\bibitem{DBLP:journals/combinatorica/DujmovicEHMW22}
Dujmovic, V., Eppstein, D., Hickingbotham, R., Morin, P., Wood, D.R.:
  Stack-number is not bounded by queue-number. Comb.  \textbf{42}(2),  151--164
  (2022). \doi{10.1007/s00493-021-4585-7}

\bibitem{DBLP:journals/jgaa/DujmovicF18}
Dujmovi{\'{c}}, V., Frati, F.: Stack and queue layouts via layered separators.
  J. Graph Algorithms Appl.  \textbf{22}(1),  89--99 (2018).
  \doi{10.7155/jgaa.00454}

\bibitem{DJMMUW20}
Dujmovi{\'c}, V., Joret, G., Micek, P., Morin, P., Ueckerdt, T., Wood, D.R.:
  Planar graphs have bounded queue-number. Journal of the ACM (JACM)
  \textbf{67}(4),  1--38 (2020). \doi{10.1145/3385731}

\bibitem{DMW05}
Dujmovi{\'c}, V., Morin, P., Wood, D.R.: Layout of graphs with bounded
  tree-width. SIAM Journal on Computing  \textbf{34}(3),  553--579 (2005).
  \doi{10.1137/S0097539702416141}

\bibitem{WPD04}
Dujmovic, V., P{\'{o}}r, A., Wood, D.R.: Track layouts of graphs. Discret.
  Math. Theor. Comput. Sci.  \textbf{6}(2),  497--522 (2004)

\bibitem{DBLP:conf/gd/DujmovicW03}
Dujmovic, V., Wood, D.R.: Three-dimensional grid drawings with sub-quadratic
  volume. In: Liotta, G. (ed.) Graph Drawing 2003. Lecture Notes in Computer
  Science, vol.~2912, pp. 190--201. Springer (2003).
  \doi{10.1007/978-3-540-24595-7\_18}

\bibitem{DBLP:journals/dmtcs/FelsnerHKO10}
Felsner, S., Huemer, C., Kappes, S., Orden, D.: Binary labelings for plane
  quadrangulations and their relatives. Discret. Math. Theor. Comput. Sci.
  \textbf{12}(3),  115--138 (2010)

\bibitem{DBLP:journals/siamdm/HeathLR92}
Heath, L.S., Leighton, F.T., Rosenberg, A.L.: Comparing queues and stacks as
  mechanisms for laying out graphs. {SIAM} J. Discret. Math.  \textbf{5}(3),
  398--412 (1992). \doi{10.1137/0405031}

\bibitem{HR92}
Heath, L.S., Rosenberg, A.L.: Laying out graphs using queues. SIAM Journal on
  Computing  \textbf{21}(5),  927--958 (1992). \doi{10.1137/0221055}

\bibitem{DBLP:conf/gd/MerkerU20}
Merker, L., Ueckerdt, T.: The local queue number of graphs with bounded
  treewidth. In: Auber, D., Valtr, P. (eds.) Graph Drawing and Network
  Visualization 2020. Lecture Notes in Computer Science, vol. 12590, pp.
  26--39. Springer (2020). \doi{10.1007/978-3-030-68766-3\_3}

\bibitem{DBLP:journals/algorithmica/Morin21}
Morin, P.: A fast algorithm for the product structure of planar graphs.
  Algorithmica  \textbf{83}(5),  1544--1558 (2021).
  \doi{10.1007/s00453-020-00793-5}

\bibitem{bob}
Pupyrev, S.: A {SAT}-based solver for constructing optimal linear layouts of
  graphs, source code available at \url{https://github.com/spupyrev/bob}

\bibitem{Pup18}
Pupyrev, S.: Mixed linear layouts of planar graphs. In: Frati, F., Ma, K.L.
  (eds.) Graph Drawing and Network Visualization. pp. 197--209. Springer
  International Publishing, Cham (2018). \doi{10.1007/978-3-319-73915-1\_17}

\bibitem{Pup20}
Pupyrev, S.: Improved bounds for track numbers of planar graphs. Journal of
  Graph Algorithms and Applications  \textbf{24}(3),  323--341 (2020).
  \doi{10.7155/jgaa.00536}

\bibitem{DBLP:journals/combinatorics/UeckerdtWY22}
Ueckerdt, T., Wood, D.R., Yi, W.: An improved planar graph product structure
  theorem. Electron. J. Comb.  \textbf{29}(2) (2022). \doi{10.37236/10614}

\bibitem{Wie17}
Wiechert, V.: On the queue-number of graphs with bounded tree-width. Electr. J.
  Comb.  \textbf{24}(1),  P1.65 (2017). \doi{10.37236/6429}

\bibitem{DBLP:journals/corr/abs-2208-10074}
Wood, D.R.: Product structure of graph classes with strongly sublinear
  separators. CoRR  \textbf{abs/2208.10074} (2022).
  \doi{10.48550/arXiv.2208.10074}

\bibitem{DBLP:journals/jcss/Yannakakis89}
Yannakakis, M.: Embedding planar graphs in four pages. J. Comput. Syst. Sci.
  \textbf{38}(1),  36--67 (1989). \doi{10.1016/0022-0000(89)90032-9}

\bibitem{DBLP:journals/jctb/Yannakakis20}
Yannakakis, M.: Planar graphs that need four pages. J. Comb. Theory, Ser. {B}
  \textbf{145},  241--263 (2020). \doi{10.1016/j.jctb.2020.05.008}

\end{thebibliography}
\end{document}